\documentclass[11pt]{amsart}
\usepackage{amsmath,mathtools}
\usepackage{graphicx, color, enumerate}
\usepackage{amscd}
\usepackage{amsmath}
\usepackage{amsfonts}
\usepackage{amssymb}
\usepackage{mathrsfs}
\usepackage{latexsym}
\newcommand\e\varepsilon

\newcommand\R{\mathbb R}

\newcommand\de\partial
\newcommand\weakto\rightharpoonup

\renewcommand\le\leqslant
\renewcommand\ge\geqslant
\renewcommand\a\alpha

\renewcommand\b\beta

\renewcommand\d\delta
\newcommand\vfi\varphi
\newcommand\g\gamma
\newcommand\gb\gamma
\renewcommand\l\lambda
\newcommand\n\nabla
\newcommand\s\sigma
\renewcommand\t\theta
\renewcommand\O\S

\newcommand\G\Gamma

\renewcommand\S\Sigma
\renewcommand\L\Lambda

\renewcommand\o\S

\def\bbm[#1]{\text{\boldmath $#1$}}
\newcommand\beq{\begin{equation}}
\newcommand\eeq{\end{equation}}

\renewcommand\leq{\leqslant}

\newtheorem{theorem}{Theorem}[section]
\newtheorem{lemma}[theorem]{Lemma}

\newtheorem{definition}[theorem]{Definition}
\newtheorem{proposition}[theorem]{Proposition}
\newtheorem{remark}[theorem]{Remark}

\def\sideremark#1{\ifvmode\leavevmode\fi\vadjust{\vbox to0pt{\vss
\hbox to0pt{\hskip\hsize\hskip1em
\vbox{\hsize3cm\tiny\raggedright\pretolerance10000
\noindent #1\hfill}\hss}\vbox to8pt{\vfil}\vss}}}

\definecolor{edu}{rgb}{0,1,0.2}

\numberwithin{equation}{section}
\title[Bound  and ground states for an elliptic system with double criticality]
{Existence of bound and ground states for an elliptic system with double criticality}


\author{Eduardo Colorado, Rafael L\'opez-Soriano, Alejandro Ortega}

\email[Eduardo Colorado ]{ecolorad@math.uc3m.es}%
\email[Rafael López-Soriano]{ralopezs@math.uc3m.es}%
\email[Alejandro Ortega ]{alortega@math.uc3m.es}
\address[E. Colorado, R. López-Soriano, A. Ortega]{Departamento de Matem\'aticas,
Universidad Carlos III de Madrid, Av. Universidad 30, 28911 Legan\'es (Madrid), Spain}

\begin{document}
\maketitle

\begin{center}
{\it   Dedicated to Ireneo Peral in memoriam}
\end{center}

\begin{abstract}
We study the existence of bound and ground states for a class of nonlinear elliptic systems in $\R^N$. These equations involve critical power nonlinearities and Hardy-type singular potentials, coupled by a term containing up to critical powers. More precisely, we find ground states either the positive coupling parameter $\nu$ is large or $\nu$ is small under suitable assumptions on the other parameters of the problem. Furthermore, bound states are found as Mountain--Pass-type critical points of the underlying functional constrained on the Nehari manifold. Our variational approach improves some known results and allows us to cover range of parameters which have not been considered previously.

\end{abstract}

\section{Introduction}
\setcounter{equation}0

\

The aim of this paper is to study a system of coupled nonlinear Schr\"ordinger equations, typically known as Gross-Pitaevskii,
\begin{equation}\label{GPsystem}
\left\{\begin{array}{ll}
\displaystyle{-i \frac{\partial}{\partial t} \Psi_1=  \Delta \Psi_1 -a_1(x) \Psi_1 + \mu_1 |\Psi_1|^2\Psi_1 +\nu |\Psi_2|^2 \Psi_1 }  & \text{in }\mathbb{R}^N, \quad t>0, \vspace{.3cm}\\
\displaystyle{-i \frac{\partial}{\partial t} \Psi_2=  \Delta \Psi_2  -a_2(x) \Psi_2 + \mu_2 |\Psi_2|^2\Psi_2 +\nu |\Psi_1|^2 \Psi_2 }  & \text{in }\mathbb{R}^N, \quad t>0,  \vspace{.3cm}\\
\displaystyle{ \Psi_j=\Psi_j(x,t)\in\mathbb{C}}, & \mbox{ for } j=1,2,     \vspace{.3cm}\\
\displaystyle{ \Psi_j(x,t)\to 0}, \qquad \mbox{ as } |x|\to+\infty,  &  \quad t>0,
\end{array}\right.
\end{equation}
where $a_1,a_2$ are real potenctials, $\mu_1,\mu_2>0$ and $\nu\neq 0$ is the so-called coupling parameter. Problem \eqref{GPsystem} arises in several physical contexts, such as the Hartree-Fock theory for a double condensate, namely a binary mixture Bose--Einstein condensates in two different hyperfine states. Here $\Psi_j$ represents the respective condensate amplitudes, the constants $\mu_j$ the interspecies lengths and the sign of $\nu$ determines the interaction between the states. See \cite{Esry, Frantz} for further details.

System \eqref{GPsystem} also appears in nonlinear optics. More precisely, it allows us to study the propagation of pulses in birefringnent optical fibers and the beam in Kerr-like photorefracive media. For more information, we refer the reader to \cite{Akhmediev,Kivshar} and references therein.

\

Let us focus on solitary-wave solutions of system \eqref{GPsystem}, namely, solutions $\Psi_1(x,t)= e^{i \lambda_1 t}u(x) $ and $\Psi_2(x,t)= e^{i \lambda_2 t}v(x)$, with $u,v$ solving
\beq\label{BSsystem}
\left\{\begin{array}{ll}
\displaystyle{-\Delta u + V_1(x) u= \mu_1 u^{3} + \nu u v^{2}}  &\text{in }\mathbb{R}^N,\vspace{.3cm}\\
-\Delta v + V_2(x) v = \mu_2 v^{3} + \nu u^{2}v &\text{in }\mathbb{R}^N,
\end{array}\right.
\eeq
where $V_1(x)=a_1(x)+\lambda_1$  and $V_2(x)=a_2(x)+\lambda_2$ with respect to \eqref{GPsystem}. This problem is typically known in the literature as the Bose-Einstein condensate system. Actually, system \eqref{BSsystem} can be seen as a particular case of the following extension
\beq\label{pBSsystem}
\left\{\begin{array}{ll}
\displaystyle{-\Delta u + V_1(x) u= \mu_1 u^{2p-1} + \nu u^{p-1} v^{p}}  &\text{in }\mathbb{R}^N,\vspace{.3cm}\\
-\Delta v + V_2(x) v = \mu_2 v^{2p-1} + \nu u^{p}v^{p-1} &\text{in }\mathbb{R}^N,
\end{array}\right.
\eeq
where $1<p\le\frac{N}{N-2}$ with $N\ge 3$. In the subcritical regime, namely $p<\frac{N}{N-2}$, the question of the existence and multiplicity of solutions for \eqref{pBSsystem} has been extensively analyzed under some assumptions on $V_j$ and $\nu$, see \cite{AC2,ACR, BW, LinWei, LW, MMP, POMP, SIR}, among others.

Regarding the critical case, $p=\frac{N}{N-2}$, if one takes $V_j$ as non-zero constants, due to a Pohozaev--type identity, system \eqref{pBSsystem} admits only the trivial solution $(u,v)=(0,0)$ . Hence, in this paper we shall assume the potential as a Hardy-type one $V_j=-\frac{\lambda_j}{|x|^2}$. This kind of potentials appears also in several physical models, for instance in nonrelativistic quantum mechanics, molecular physics or combustion models. Notice that for \eqref{pBSsystem} with $V_j=-\frac{\lambda_j}{|x|^2}$ the non-trivial solutions are singular around the point $x=0$.

Concerning systems with singular potentials, we point out the following Hamiltonian system introduced in \cite{FigPerRos},
\begin{equation*}
\left\{\begin{array}{ll}
\displaystyle -\Delta u = \frac{v^p}{|x|^{\alpha}}   &\text{in }\Omega,\vspace{.3cm}\\
\displaystyle -\Delta v =  \frac{u^q}{|x|^{\beta}} &\text{in }\Omega,\vspace{.3cm}\\
\end{array}\right.
\end{equation*}
where $\Omega\subset \mathbb{R}^N$ is a bounded domain and some conditions on $p,q,\alpha,\beta$ are imposed. See also \cite{BVG} where some estimates on the singularity of the solutions near the origin where obtained previously for Hamiltonian systems with absorption terms.

\

In this work, we will consider the problem
\beq\label{system:alphabeta}
\left\{\begin{array}{ll}
\displaystyle -\Delta u - \lambda_1 \frac{u}{|x|^2}-u^{2^*-1}= \nu \alpha h(x) u^{\alpha-1}v^\beta   &\text{in }\mathbb{R}^N,\vspace{.3cm}\\
\displaystyle -\Delta v - \lambda_2 \frac{v}{|x|^2}-v^{2^*-1}= \nu \beta h(x) u^\alpha v^{\beta-1} &\text{in }\mathbb{R}^N,\vspace{.3cm}\\
u,v> 0 & \text{in }\mathbb{R}^N\setminus\{0\},
\end{array}\right.
\eeq
where $\lambda_1,\lambda_2\in(0,\Lambda_N)$ with $\Lambda_N=\frac{(N-2)^2}{4}$ the best constant in the Hardy's inequality, $2^*=\frac{2N}{N-2}$ is the critical Sobolev exponent, $\nu>0$, $\alpha,\beta$ are real parameters and $h$ is a certain function defined in $\mathbb{R}^N$. In particular, we shall assume
\beq\label{alphabeta}\tag{$\alpha\beta$}
\alpha, \beta> 1 \qquad \mbox{ and } \qquad  \alpha+\beta\le2^*,
\eeq
and
\beq\label{H}\tag{H}
h \mbox{ is a positive $ L^\infty(\mathbb{R}^N) $ function. }
\eeq

Note that system \eqref{system:alphabeta} is related to \eqref{pBSsystem} with $\mu_j=1$ for $j=1,2$ and the potentials are of Hardy kind ones.

\

Our goal is to derive the existence of positive solutions to the system \eqref{system:alphabeta}. We shall obtain the main results of the work by means of variational methods. More precisely, we will look for solutions as critical points of the associated energy functional
\begin{equation}\label{functalphabeta}
\begin{split}
\mathcal{J}_\nu (u,v)=&\frac{1}{2} \int_{\mathbb{R}^N} \left( |\nabla u|^2 + |\nabla v|^2  \right) \, dx -\frac{\lambda_1}{2} \int_{\mathbb{R}^N} \dfrac{u^2}{|x|^2} \, dx   -\frac{\lambda_2}{2} \int_{\mathbb{R}^N} \dfrac{v^2}{|x|^2} \, dx  \vspace{0,7cm} \\
&- \frac{1}{2^*} \int_{\mathbb{R}^N} \left( |u|^{2^*} + |v|^{2^*}  \right) \, dx  -\nu \int_{\mathbb{R}^N} h(x) |u|^\alpha |v|^{\beta} \, dx  ,
\end{split}
\end{equation}
defined in the the Sobolev--based space  $\mathbb{D}=\mathcal{D}^{1,2} (\mathbb{R}^N)\times \mathcal{D}^{1,2}(\mathbb{R}^N)$. Here we consider the space $\mathcal{D}^{1,2}(\mathbb{R}^N)$ as the completion of $C^{\infty}_0(\mathbb{R}^N)$ under the norm
$$
\|u\|_{\mathcal{D}^{1,2}(\mathbb{R}^N)}=\left(\int_{\mathbb{R}^N} \, |\nabla u|^2  \, dx  \right)^{1/2}.
$$

A crucial point in our analysis will be the played by the \textit{semi-trivial} solutions, namely couples of solutions with a trivial component. Observe that for any $\nu \in \mathbb{R}$, problem \eqref{system:alphabeta} admits two \textit{semi-trivial} positive solutions $(z_1,0)$ and $(0,z_2)$, where $z_j$ satisfies the entire problem
$$
-\Delta z_j - \lambda_j \frac{z_j}{|x|^2}=z_j^{2^*-1} \qquad \mbox{ and } \qquad z_j>0 \qquad \mbox{ in } \mathbb{R}^N\setminus \{ 0\},
$$
for $j=1,2$. The explicit expression of $z_j$ was found by Terracini in \cite{Terracini}, which is recalled in Section~\ref{section2} joint with several energy properties.

\

A characterization of the \textit{semi-trivial} solutions as critical points of the energy functional is provided by Abdellaoui, Felli and Peral (cf. \cite{AbFePe}). In fact, the authors prove that the couples $(z_1,0)$ and $(0,z_2)$ become either a local minimum or a saddle point of $\mathcal{J}_\nu$ on the corresponding Nehari manifold (see Section~\ref{section2}) under some hypotheses on the parameters $\nu,\alpha,\beta$. This classification will be crucial in order to study the geometry of the functional $\mathcal{J}_\nu$ and to obtain some energy estimates, which will allow us to deduce existence of solutions. For that reason, we recall the classification in Proposition~\ref{thmsemitrivialalphabeta}.

The coupling parameter $\nu$ and the exponents $\alpha,\beta$ dramatically affect the behavior of the functional $\mathcal{J}_\nu$. In particular, if $\nu$ is large enough, one can find a positive ground state (cf. \cite{AbFePe}), that we also include in Theorem~\ref{thm:nugrande} to have a complete picture of the results. Conversely, if $\alpha,\beta\ge 2$ and $\nu$ is small enough, the ground states correspond to the \textit{semi-trivial} solutions (cf. \cite{AbFePe}). Moreover, under this assumption, the functional $\mathcal{J}_\nu$ restricted on the Nehari manifold $\mathcal{N}_\nu$ exhibits a Mountain--Pass geometry. Since the Palais-Smale (PS for short) condition is satisfied under certain hypotheses on $h$ and $\alpha,\beta$, the existence of bound states is automatically guaranteed.

The critical case, $\alpha+\beta=2^*$, was widely analyzed by Chen and Zou in \cite{ChenZou}. Among other questions, they establish that the lower bound of $\mathcal{J}_\nu$ on $\mathcal{N}_\nu$ is not reached for $\nu$ strictly negative. Moreover, they find existence of radial ground states for $h(x)=1$ under different premises on $\alpha,\beta$ and $\lambda_1,\lambda_2$, as well as they describe the limit energy level of the solutions as $\nu\to0$. The case of a sign-changing potential was studied in \cite{ZhongZou} under somehow optimal hypotheses on the potential $h$.

We point out that the case in which one of the exponents $\alpha,\beta$ may be $1$ is considered in a forthcoming paper, \cite{CoLSOr}. For that setting, one of the \textit{semi-trivial} couples is no longer solution of \eqref{system:alphabeta}. Indeed, the authors focus on the special case in which $\alpha=2$ and $\beta=1$. This configuration appears in the so-called Schr\"odinger-Korteweg-de Vries which models some phenomena in fluid mechanics, see \cite{AA, Col, DFO, FO} and references therein for more details.

To the best of our knowledge, the cases in which $1<\alpha<2\leq \beta$ and $\lambda_2<\lambda_1$ or $1<\beta<2\leq \alpha$ and $\lambda_1<\lambda_2$ are not considered by the previous works. Hence, the main goal of this work is to cover these ranges.

In case that $1<\beta<2\leq \alpha$ and $\nu$ is small enough, the couple $(0,z_2)$ becomes a local minimum and $(z_1,0)$ is a saddle point of $\mathcal{J}_\nu$. The order between the parameters $\lambda_1$ and $\lambda_2$ determines the order between the \textit{semitrivial} energy levels. Indeed, if $\lambda_2>\lambda_1$ with $\nu$ small enough, the couple $(0,z_2)$ corresponds to the ground state of \eqref{system:alphabeta}, see Theorem~\ref{thm:groundstatesalphabeta}. Assuming that the parameters are somehow closed, it is proved that the energy functional exhibits a Mountain--Pass geometry and, consequently, a positive bound state is found. We refer to Theorem~\ref{MPgeom} for further details. Analogous conclusions are obtained in case that \break $1<\alpha<2\leq\beta$ and $\lambda_1>\lambda_2$.

The main ingredient in our approach is the application of an algebraic lemma in the component whose exponent is greater than $2$, see Lemma~\ref{algelemma}. More precisely, this result allows us to derive lower bounds on integral terms and, consequently, to prove the critical mass of the underlying component must not vanish. For that reason, under suitable hypotheses on $\nu$, we find ground and bound states supposing that $\max \{\alpha,\beta\}\ge 2$.

\

As commented previously, to establish the existence results, first we need to guarantee some compactness properties. This feature is given by a PS condition, relying on the well known \textit{concentration-compactness principle}, cf. \cite{Lions1,Lions2}. In this context, we shall deal with the lack of compactness of the embedding of the space $\mathcal{D}^{1,2}(\mathbb{R}^N)$ in $L^{2^*}(\mathbb{R}^N)$. Then, the nonlinear coupling term, $u^\alpha v^\beta$, may be critical depending on the values of the exponents. Indeed, we will consider two different cases between the subcritical exponents ($\alpha+\beta< 2^*$) and the critical ones ($\alpha+\beta= 2^*$). If $\alpha+\beta< 2^*$, compactness follows by standard embedding, whereas if $\alpha+\beta=2^*$ the issue is more delicate and we need a careful analysis and extra assumptions on $h$.

\

Organization of the paper: In Section \ref{section2}, we introduce notation, we give the definitions of bound and ground states, we introduce the Nehari Manifold and we show the character of the semi-trivial solutions.
Section \ref{section:PS} is devoted to prove the PS condition in both regimes, the subcritical and critical ones. In Section \ref{section:main} we show the main results on the existence of bound and ground states of \eqref{system:alphabeta}.

\

\section{Preliminaries and Functional setting}\label{section2}

In this section, we present a suitable variational setting for the system \eqref{system:alphabeta}. Actually, the problem \eqref{system:alphabeta} is the Euler-Lagrange system for the energy functional $\mathcal{J}_\nu$, introduced in \eqref{functalphabeta}, defined in the product space $\mathbb{D}=\mathcal{D}^{1,2} (\mathbb{R}^N)\times \mathcal{D}^{1,2} (\mathbb{R}^N)$. The energy space $\mathbb{D}$ is equipped with the norm
\begin{equation*}
\|(u,v)\|^2_{\mathbb{D}}=\|u\|^2_{\lambda_1}+\|v\|^2_{\lambda_2},
\end{equation*}
where
\begin{equation*}
\|u\|^2_{\lambda}=\int_{\mathbb{R}^N} |\nabla u|^2 \, dx - \lambda \int_{\mathbb{R}^N} \frac{u^2}{|x|^2} \, dx.
\end{equation*}

\

Let us point out that, because of the Hardy's inequality,
\begin{equation}\label{hardy_inequality}
\Lambda_N \int_{\R^N} \dfrac{u^2}{|x|^2} \, dx \leq \int_{\R^N} |\nabla u|^2 \, dx,
\end{equation}
the norm $\|\cdot\|_{\lambda}$ is equivalent to the norm $\|\cdot\|_{\mathcal{D}^{1,2} (\mathbb{R}^N)}$ for any $\lambda\in (0,\Lambda_N)$, where $\Lambda_N=\frac{(N-2)^2}{4}$ is the optimal constant in the Hardy inequality.

\

As mentioned above, the particular case of a single equation has been extensively studied. In particular, if either the system is decoupled, namely $\nu=0$, or some component vanishes, $u$ or $v$ satisfies the entire equation
\beq\label{entire}
-\Delta z - \lambda \frac{z}{|x|^2}=z^{2^*-1} \qquad\mbox{ and }   \qquad z>0\ \mbox{ in } \mathbb{R}^N \setminus \{0\}.
\eeq
A complete classification of \eqref{entire} was carried out by Terracini in \cite{Terracini}. Indeed, it is proved that if $\lambda\in\left(0,\Lambda_N\right)$, then the solutions of the problem \eqref{entire} takes the expression
\beq\label{zeta}
z_\mu^{\lambda}(x)= \mu^{-\frac{N-2}{2}}z_1^{\lambda}\left(\frac{x}{\mu}\right) \qquad \mbox{ with } \qquad z_1^{\lambda}(x)=\dfrac{A(N,\lambda)}{|x|^{a_\lambda}\left(1+|x|^{2-\frac{4a_\lambda}{N-2}}\right)^{\frac{N-2}{2}}},
\eeq
where $a_\lambda=\frac{N-2}{2}-\sqrt{\left( \frac{N-2}{2}\right)^2-\lambda}$, $A(N,\lambda)=\frac{N(N-2-2a_\lambda)^2}{N-2} $ and $\mu>0$ is a scaling factor. Solutions of \eqref{entire} arise as minimizers of the underlying Rayleigh quotient
\beq\label{Slambda}
\mathcal{S}(\lambda)= \inf_{\substack{u\in \mathcal{D}^{1,2}(\mathbb{R}^N)\\
u\not\equiv0}}\frac{\|u\|^2_{\lambda}}{\|u_\mu^\lambda\|_{2^*}^{2}}= \frac{\|z_\mu^\lambda\|^2_{\lambda}}{\|z_\mu^\lambda\|_{2^*}^{2}}= \left(1-\frac{4\lambda}{(N-2)^2} \right)^{\frac{N-1}{N}}\mathcal{S},
\eeq
where $\mathcal{S}$ denotes the optimal constant in Sobolev's inequality
\begin{equation}\label{sobolev_inequality}
\mathcal{S}\int_{\R^N} |u|^{2^*}dx \leq \int_{\R^N} |\nabla u|^2dx.
\end{equation}

By explicit computations, it holds that
\beq\label{normcrit}
\displaystyle  \|z_\mu^\lambda\|_{\lambda}^{2} = \|z_\mu^\lambda\|_{L^{2^*}}^{2^*}= \mathcal{S}^{\frac{N}{2}}(\lambda).
\eeq

Therefore, for every $\mu>0$ the couples $(z_\mu^{\lambda_1},0)$ and $(0,z_\mu^{\lambda_2})$ satisfy \eqref{system:alphabeta}. From now on, we will refer to these kind of solutions as \textit{semi-trivial} solutions. Our main objective is to look for solutions neither \textit{semi-trivial} nor trivial, i.e., couples of solutions $(u,v)$ such that $u\not\equiv 0$ and $v\not\equiv 0$ in $\mathbb{R}^N$.
\begin{definition}
We say that $(u,v)\in\mathbb{D}\setminus (0,0)$,  is a non-trivial {\it bound state} of \eqref{system:alphabeta} if $(u,v)$ is a non-trivial  critical point
of $\mathcal{J}_\nu$.
A bound state $(\tilde{u},\tilde{v})$ is called ground state if its energy is minimal among all the non-trivial and non-negative bound states, namely
\begin{equation}\label{ctilde}
\tilde{c}_\nu=\mathcal{J}_\nu(\tilde{u},\tilde{v})=\min\left\{\mathcal{J}_\nu(u,v): (u,v)\in \mathbb{D}\setminus (0,0),\; u,v\ge 0, \mbox{ and } \mathcal{J}_\nu'(u,v)=0\right\}.
\end{equation}
\end{definition}
Let us stress that the functional $\mathcal{J}_\nu \in C^1(\mathbb{D},\mathbb{R})$. In addition, the energy functional is not bounded from below. Indeed, given $(\tilde u,\tilde v)\in \mathbb{D}$, we have
\begin{equation*}
\mathcal{J}_\nu(t \tilde u,t \tilde v) \to -\infty \qquad \mbox{ as } t\to\infty,
\end{equation*}

In order to minimize the energy functional, it shall be convenient to introduce a proper constraint. Let us introduce the Nehari manifold associated to the functional $\mathcal{J}_\nu$, denoted by $\mathcal{N}_\nu$, and defined as
\begin{equation*}
\mathcal{N}_\nu=\left\{ (u,v) \in \mathbb{D} \setminus (0,0) \, : \,  \Phi_\nu(u,v)=0 \right\},
\end{equation*}
where
\beq\label{Psi}
\displaystyle \Phi_\nu(u,v)=\left\langle \mathcal{J}'_\nu(u,v){\big|}(u,v)\right\rangle.
\eeq
Plainly, the set $\mathcal{N}_\nu$ contains all the non-trivial critical points of $\mathcal{J}_\nu$ in $\mathbb{D}$. For the reader's convenience, we will recall some well-known facts about Nehari manifolds.

Given $(u,v) \in \mathcal{N}_\nu$, then it holds that
\begin{equation} \label{Nnueq1}
 \|(u,v)\|_\mathbb{D}^2=\|u_n\|_{L^{2^*}}^{2^*}+\|v_n\|_{L^{2^*}}^{2^*} +\nu (\alpha+\beta) \int_{\mathbb{R}^N} h(x) |u|^{\alpha} |v|^{\beta} \, dx,
\end{equation}
so the restricted energy functional can be rewritten as
\beq\label{Nnueq2}
\mathcal{J}_{\nu}{\big|}_{\mathcal{N}_\nu} (u,v) = \frac{1}{N} \int_{\mathbb{R}^N} \left(  |u|^{2^*} +  |v|^{2^*}  \right)  +\nu \left( \frac{\alpha+\beta-2}{2} \right)  \int_{\mathbb{R}^N} h(x) |u|^{\alpha} |v|^{\beta}   \, dx.
\eeq

For every $(u,v)\in\mathbb{D}\setminus \{(0,0)\}$, there exists a unique $t=t_{(u,v)}$ such that $(tu,tv) \in \mathcal{N}_\nu$. Indeed, $t_{(u,v)}$ can be defined as the unique solution to the algebraic equation
\beq\label{normH}
\|(u,v)\|_\mathbb{D}^2=t^{2^*-2}  \left( \|u_n\|_{L^{2^*}}^{2^*}+\|v_n\|_{L^{2^*}}^{2^*} \right) + \nu (\alpha+\beta) \, t^{\alpha+\beta-2} \int_{\mathbb{R}^N} h(x) |u|^{\alpha} |v|^{\beta}   \, dx.
\eeq
By using \eqref{Nnueq1} and hypotheses \eqref{alphabeta}, one gets that, for any $(u,v) \in \mathcal{N}_\nu$,
\beq\label{criticalpoint1}
\begin{split}
\mathcal{J}_\nu''(u,v)[u,v]^2&=\left\langle \Phi_\nu'(u,v){\big|}(u,v)\right\rangle\\
&=(2-\alpha-\beta)\|(u,v)\|_{\mathbb{D}}^2+(\alpha+\beta-2^*)\left( \|u\|_{L^{2^*}}^{2^*}+\|v\|_{L^{2^*}}^{2^*} \right) <0.
\end{split}
\eeq
Thus, $\mathcal{N}_\nu$ is a locally smooth manifold near any $(u,v)\in \mathbb{D}\setminus \{(0,0)\}$ with $\Phi_\nu(u,v)=0$. Moreover,  the second variation of the functional $\mathcal{J}_{\nu}$ at $(0,0)$ along the direction $(\varphi_1,\varphi_2)$ satisfies
\begin{equation*}
 \mathcal{J}_\nu''(0,0)[\varphi_1,\varphi_2]^2=\|(\varphi_1,\varphi_2)\|^2_{\mathbb{D}}>0 \quad \text{ for any } (\varphi_1,\varphi_2)\in \mathcal{N}_\nu,
\end{equation*}
thus, we can infer that $(0,0)$ is a strict minimum for $\mathcal{J}_\nu$. As a consequence, $(0,0)$ is an isolated point of the set $\displaystyle \mathcal{N}_\nu  \, \cup \, (0,0)$ and, therefore, the Nehari manifold $\mathcal{N}_\nu$ is a smooth complete manifold of codimension $1$. Moreover, there exists a constant $r_\nu>0$ such that
\beq\label{criticalpoint2}
\|(u,v)\|_{\mathbb{D}} > r_\nu\quad\text{for all } (u,v)\in \mathcal{N}_\nu.
\eeq

On the other hand, let $(u,v) \in \mathbb{D}$ be a critical point of $\mathcal{J}_\nu$ constrained on $\mathcal{N}_\nu$, then there exists a Lagrange multiplier $\omega$ such that
\begin{equation*}
(\mathcal{J}_{\nu}{\big|}_{\mathcal{N}_\nu})'(u,v)=\mathcal{J}'_\nu(u,v)-\omega \Phi'_\nu(u,v)=0.
\end{equation*}
Consequently, we obtain that $\displaystyle \Phi_\nu(u,v)=\left\langle \mathcal{J}'_\nu(u,v){\big|}(u,v)\right\rangle = \omega \left\langle \Phi'_\nu(u,v){\big|}(u,v)\right\rangle $. By \eqref{criticalpoint1}, one finds that $\omega=0$ and $\mathcal{J}'_\nu(u,v)=0$. To summarize,
\vspace{0.15cm}
\begin{center}
$(u,v) \in \mathbb{D}$ is a critical point of $\mathcal{J}_\nu$ $\quad\Leftrightarrow\quad$ $(u,v) \in \mathbb{D}$ is a critical point of $\mathcal{J}_{\nu}{\big|}_{\mathcal{N}_\nu}$.
\end{center}
\vspace{0.15cm}

Let us also note that, on the Nehari manifold $\mathcal{N}_\nu$,
\beq\label{Nnueq}
(\mathcal{J}_{\nu}{\big|}_{\mathcal{N}_\nu})' (u,v) = \left( \frac{1}{2}-\frac{1}{\alpha+\beta} \right) \|(u,v)\|^2_{\mathbb{D}} +  \left( \frac{1}{\alpha+\beta}-\frac{1}{2^*} \right) \left( \|u\|_{L^{2^*}}^{2^*}+\|v\|_{L^{2^*}}^{2^*} \right).
\eeq
As a consequence of hypotheses \eqref{alphabeta} and \eqref{criticalpoint2}, we have
\begin{equation*}
\mathcal{J}_{\nu} (u,v) > \left( \frac{1}{2}-\frac{1}{\alpha+\beta} \right) r^2_\nu\quad  \quad\text{for all } (u,v)\in \mathcal{N}_\nu.
\end{equation*}
Therefore, the energy functional $\mathcal{J}_{\nu} (u,v)$ is bounded from below on $\mathcal{N}_\nu$, so we can try to obtain solutions of \eqref{system:alphabeta} by minimizing $\mathcal{J}_\nu$  on the Nehari manifold.

\subsection{Semi-trivial solutions}

\

As commented in the in the introduction, a particular type of solutions are the \textit{semi-trivial} ones which will help us to obtain certain energy estimates useful along our study. To do so, first we determine their variational nature.

Let us consider the decoupled energy functionals $\mathcal{J}_j:\mathcal{D}^{1,2} (\mathbb{R}^N)\mapsto\mathbb{R}$,
\begin{equation}\label{funct:Ji}
\mathcal{J}_j(u) =\frac{1}{2} \int_{\mathbb{R}^N}  |\nabla u|^2 \, dx -\frac{\lambda_j}{2} \int_{\mathbb{R}^N} \dfrac{u^2}{|x|^2}  \, dx - \frac{1}{2^*} \int_{\mathbb{R}^N} |u|^{2^*} \, dx,
\end{equation}
for $j=1,2$. Note that
\begin{equation*}
\mathcal{J}_\nu(u,v)=\mathcal{J}_1(u)+\mathcal{J}_2(v)-\nu \int_{\mathbb{R}^N} h(x) |u|^{\alpha} |v|^{\beta} \, dx.
\end{equation*}
The function $z_\mu^{\lambda_j}$, defined in \eqref{zeta}, is a global minimum of $\mathcal{J}_j$ on the underlying Nehari manifold
\begin{equation*}
\begin{split}
\mathcal{N}_j&= \left\{ u \in \mathcal{D}^{1,2} (\mathbb{R}^N) \setminus \{0\} \, : \,  \left\langle \mathcal{J}'_j(u){\big|} u\right\rangle=0 \right\}\\
&= \left\{ u \in\mathcal{D}^{1,2} (\mathbb{R}^N) \setminus \{0\} \, : \,  \|u\|_{\lambda_j}=\|u\|_{L^{2^*}}^{2^*} \, dx \right\}.
\end{split}
\end{equation*}

By a direct computation, it can be checked that the energy levels of $z_\mu^{\lambda_j}$, and therefore the \textit{semi-trivial solutions} energy, are given by
\beq\label{Jzeta}
\mathcal{J}_1(z_\mu^{\lambda_1})=\dfrac{1}{N}\mathcal{S}^{\frac{N}{2}}(\lambda_1)=\mathcal{J}_\nu(z_\mu^{\lambda_1},0), \qquad \mathcal{J}_2(z_\mu^{\lambda_2})=\dfrac{1}{N}\mathcal{S}^{\frac{N}{2}}(\lambda_2)=\mathcal{J}_\nu(0,z_\mu^{\lambda_2}),
\eeq
for any $\mu>0$ where $\mathcal{S}(\lambda)$ is defined in \eqref{Slambda}.

\

The following result collects the behavior of the \textit{semi-trivial} solutions $(z_\mu^{\lambda_1},0)$ and $(0,z_\mu^{\lambda_2})$ concerning the energy functional $\mathcal{J}_\nu$ on the Nehari manifold $\mathcal{N}_\nu$.
\begin{proposition}\label{thmsemitrivialalphabeta}{\cite[Theorem~2.2]{AbFePe}}  Under hypotheses \eqref{alphabeta} and \eqref{H}, the following holds:
\begin{enumerate}
\item[i)] If $\alpha>2$ or $\alpha=2$ and $\nu$ small enough, then $(0,z_\mu^{\lambda_2})$ is a local minimum of $\mathcal{J}_\nu$ on $\mathcal{N}_\nu$.
\item[ii)] If $\beta>2$ or $\beta=2$ and $\nu$ small enough, then $(z_\mu^{\lambda_1},0)$ is a local minimum of $\mathcal{J}_\nu$ on $\mathcal{N}_\nu$.
\item[iii)] If $\alpha<2$ or $\alpha=2$ and $\nu$ large enough, then $(0,z_\mu^{\lambda_2})$ is a saddle point for $\mathcal{J}_\nu$ on $\mathcal{N}_\nu$.
\item[iv)] If $\beta<2$ or $\beta=2$ and $\nu$ large enough, then $(z_\mu^{\lambda_1},0)$ is a saddle point for $\mathcal{J}_\nu$ on $\mathcal{N}_\nu$.
\end{enumerate}
\end{proposition}

\

Before ending this preliminary section, let us recall an useful algebraic result, which will be used in several proofs along this paper.
\begin{lemma}\label{algelemma}{\cite[Lemma 3.3]{AbFePe}}
Let $A, B>0$ and $\gamma \ge 2$, consider the set
\begin{equation*}
\Sigma_\nu=\{\sigma \in (0,+\infty)  \, : \, A \sigma^{\frac{N-2}{N}} < \sigma + B \nu \sigma^{\frac{\gamma}{2} \frac{N-2}{N}} \}.
\end{equation*}
Then, for every $\varepsilon>0$ there exists $\tilde{\nu}>0$ such that
\begin{equation*}
\inf_{\Sigma_\nu} \sigma > (1-\varepsilon) A^{\frac{N}{2}} \qquad \mbox{ for any } 0<\nu<\tilde{\nu}.
\end{equation*}
\end{lemma}

\

\section{The Palais-Smale condition}\label{section:PS}

A crucial step to obtain the existence of solutions relies on the compactness of the energy functional $\mathcal{J}_\nu$ provided by a PS condition. Before stating several results, let us recall some general definitions concerning PS sequences and the PS condition.

\begin{definition}
Let $V$ be a Banach space. We say that $\{u_n\} \subset V$ is a PS sequence for an energy functional $\mathfrak{F}:V\mapsto\mathbb{R}$ if
\begin{equation*}
\mathfrak{F}(u_n) \to c \quad \hbox{ and }\quad  \mathfrak{F}'(u_n) \to 0\quad\mbox{in}\ V'\quad \hbox{as}\quad n\to + \infty,
\end{equation*}
where $V'$ is the dual space of $V$. Moreover, we say that $\{u_n\}$ satisfies a PS condition if
\begin{equation*}
\{u_n\}\quad \mbox{has a strongly convergent subsequence.}
\end{equation*}

In particular, we say that the functional $\mathfrak{F}$ satisfies the PS condition at level $c$ if every PS sequence at level $c$ for $\mathfrak{F}$ satisfies the PS condition.
\end{definition}

As a first step, we shall prove that a PS sequence for the energy functional $\mathcal{J}_\nu$ restricted to the Nehari manifold $\mathcal{N}_\nu$ is also a PS sequence for  $\mathcal{J}_\nu$ defined in the whole space $\mathbb{D}$.

\begin{lemma}\label{lemma:PSNehari}
Assume hypotheses \eqref{alphabeta} and \eqref{H}.  Let $\{(u_n,v_n)\} \subset \mathcal{N}_\nu$ be a PS sequence for $\mathcal{J}_\nu {\big|}_{\mathcal{N}_\nu}$ at level $c\in\mathbb{R}$. Then, $\{(u_n,v_n)\}$ is a PS sequence for $\mathcal{J}_\nu$ in $\mathbb{D}$, namely
\beq\label{PSNehari}
\mathcal{J}_{\nu}'(u_n,v_n)\to0 \quad \mbox{ in } \mathbb{D}' \quad\text{as }n\to+\infty.
\eeq
\end{lemma}

\begin{proof}
Let $\{(u_n,v_n)\}\subset \mathcal{N}_\nu$ be a PS sequence for $\mathcal{J}_{\nu}$ at level $c$, by \eqref{Nnueq}, one has
\begin{equation*}
c+o(1)=\mathcal{J}_{\nu}(u_n,v_n)\ge \left(\frac 12 - \frac{1}{\alpha+\beta} \right) \|(u_n,v_n)\|^2_{\mathbb{D}},
\end{equation*}
which implies that  $\{(u_n,v_n)\}$ is a bounded sequence in $\mathbb{D}$. Moreover, consider the functional $\Phi_\nu$ introduced in \eqref{Psi}. From inequalities \eqref{criticalpoint1} and \eqref{criticalpoint2}, we get
\begin{equation}\label{PSNehari1}
\left\langle \Phi_\nu'(u_n,v_n){\big|}(u_n,v_n)\right\rangle\\ \leq (2-\alpha-\beta)r^2_\nu.
\end{equation}
Let $\{\omega_n\}\subset \R$ be the sequence of multiplers provided by the method of Lagrange, then
\begin{equation}\label{PSNehari2}
(\mathcal{J}_\nu {\big|}_{\mathcal{N}_\nu})'(u_n,v_n) = \mathcal{J}'_\nu (u_n,v_n)-\omega_n \Phi_\nu'(u_n,v_n) \quad \mbox{ in } \mathbb{D}'.
 \end{equation}
Taking the scalar product with $(u_n,v_n)$ in the previous expression, we have that $\displaystyle \Phi_\nu(u_n,v_n)=\left\langle \mathcal{J}'_\nu(u_n,v_n){\big|}(u_n,v_n)\right\rangle = \omega_n \left\langle \Phi'_\nu(u_n,v_n){\big|}(u_n,v_n)\right\rangle =0$. By \eqref{PSNehari1}, one derives that $\omega_n\to0$. Thus, \eqref{PSNehari2} implies directly \eqref{PSNehari}.
\end{proof}

Next, we focus on boundedness of PS sequences that, together with the natural embedding of the space $\mathcal{D}^{1,2}$, will provide compactness of PS sequences.

\begin{lemma}\label{lemmaPS0}
Assume  hypotheses \eqref{alphabeta} and \eqref{H}. Let $\{(u_n,v_n)\} \subset \mathbb{D}$ be a PS sequence for $\mathcal{J}_\nu$ at level $c\in\mathbb{R}$. Then,  $\|(u_n,v_n)\|_{\mathbb{D}}<C$.
\end{lemma}

\begin{proof}
Let $\{(u_n,v_n)\}$ be a PS sequence for $\mathcal{J}_{\nu}$ at level $c$, i.e.
\begin{equation*}
\mathcal{J}_{\nu}(u_n,v_n)\to c\quad\text{and}\quad \mathcal{J}_{\nu}'(u_n,v_n)\to0\quad\text{as }n\to+\infty.
\end{equation*}
Since $\mathcal{J}_{\nu}'(u_n,v_n)\to0$ in $\mathbb{D}'$, in particular,
\begin{equation*}
\left\langle \mathcal{J}_{\nu}'(u_n,v_n)\left|\frac{(u_n,v_n)}{\|(u_n,v_n)\|_{\mathbb{D}}} \right.\right\rangle\to0.
\end{equation*}
Thus, there exists a subsequence (still denoted by $\{(u_n,v_n)\}$) such that
\begin{equation*}
\|(u_n,v_n)\|_{\mathbb{D}}^2-\left( \|u_n\|_{L^{2^*}}^{2^*}+\|v_n\|_{L^{2^*}}^{2^*} \right)-\nu(\alpha+\beta)\int_{\mathbb{R}^N}h(x)|u_n|^{\alpha}|v_n|^{\beta} \, dx=\|(u_n,v_n)\|_{\mathbb{D}}\cdot o(1).
\end{equation*}
Moreover, since $\mathcal{J}_{\nu}(u_n,v_n)\to c$, we have
\begin{equation*}
\frac12\|(u_n,v_n)\|_{\mathbb{D}}^2-\frac{1}{2^*}\left( \|u_n\|_{L^{2^*}}^{2^*}+\|v_n\|_{L^{2^*}}^{2^*} \right)-\nu\int_{\mathbb{R}^N}h(x)|u_n|^{\alpha}|v_n|^{\beta} \, dx =c+o(1).
\end{equation*}
Therefore
\begin{equation*}
\mathcal{J}_{\nu}(u_n,v_n)-\frac{1}{\alpha+\beta}\left\langle \mathcal{J}_{\nu}'(u_n,v_n)\left|\frac{(u_n,v_n)}{\|(u_n,v_n)\|_{\mathbb{D}}} \right.\right\rangle=c+\|(u_n,v_n)\|_{\mathbb{D}}\cdot o(1),
\end{equation*}
and, hence,
\begin{equation}\label{eq:limit2}
\left( \frac{1}{2}- \frac{1}{\alpha+\beta} \right)\|(u_n,v_n)\|_{\mathbb{D}}^2+\left(\frac{1}{\alpha+\beta}- \frac{1}{2^*} \right)\left( \|u_n\|_{L^{2^*}}^{2^*}+\|v_n\|_{L^{2^*}}^{2^*} \right)=c+\|(u_n,v_n)\|_{\mathbb{D}}\cdot o(1).
\end{equation}
As a consequence
$\displaystyle
\left( \frac{1}{2}- \frac{1}{\alpha+\beta} \right) \|(u_n,v_n)\|_{\mathbb{D}}^2\leq c+\|(u_n,v_n)\|_{\mathbb{D}}\cdot o(1),
$
from where we conclude that the sequence $\{(u_n,v_n)\}$ is bounded in $\mathbb{D}$.
\end{proof}

\

\subsection{Subcritical range $ \alpha+\beta < 2^*$}\hfill\newline

In the subsequent result, we establish the Palais-Smale condition for subcritical energy levels of $\mathcal{J}_\nu$, which will allow to find existence of solutions for \eqref{system:alphabeta} by minimizing the energy functional.

\begin{lemma}\label{lemmaPS2}
 Assume $\alpha+\beta<2^*$ and \eqref{H}.  Then, the functional $\mathcal{J}_\nu$ satisfies the PS condition for any level $c$ such that
\beq\label{hyplemmaPS2}
c<\frac{1}{N} \min\{ \mathcal{S}(\lambda_1),\mathcal{S}(\lambda_2) \}^{\frac{N}{2}}.
\eeq
\end{lemma}
\begin{proof}

 From Lemma~\ref{lemmaPS0}, we have that any PS sequence is bounded in $\mathbb{D}$. Therefore, there exists $(\tilde{u},\tilde{v})\in\mathbb{D}$ and a subsequence (denoted also by $\{(u_n,v_n)\}$) such that\begin{align*}\
(u_n,v_n) \rightharpoonup (\tilde{u},\tilde{v})& \quad \hbox{weakly in  } \mathbb{D},\\
(u_n,v_n) \to (\tilde{u},\tilde{v})&\quad \hbox{strongly in  } L^q(\mathbb{R}^N)\times L^q(\mathbb{R}^N)\text{ for } 1\leq q<2^*,\\
(u_n,v_n) \to (\tilde{u},\tilde{v})&\quad \hbox{a.e. in  }\mathbb{R}^N.
\end{align*}
Because of the \textit{concentration-compactness principle} by Lions (cf. \cite{Lions1,Lions2}), there exist a subsequence (still denoted by) $\{(u_n,v_n)\}$, two (at most countable) sets of points $\{x_j\}_{j\in\mathfrak{J}}\subset\mathbb{R}^N$ and $\{y_k\}_{k\in\mathfrak{K}}\subset\mathbb{R}^N$, and positive numbers $\{\mu_j,\rho_j\}_{j\in\mathfrak{J}}$, $\{\overline{\mu}_k,\overline{\rho}_k\}_{k\in\mathfrak{K}}$, $\mu_0$, $\rho_0$, $\gamma_0$, $\overline{\mu}_0$, $\overline{\rho}_0$ and $\overline{\gamma}_0$ such that, in the sense of measures,
\begin{equation}\label{con-comp}
\left\{
\begin{array}{rl}
|\nabla u_n|^2\mkern-10mu &\rightharpoonup d\mu\ge |\nabla \tilde{u}|^2+\sum_{j\in\mathfrak{J}}\mu_j\delta_{x_j}+\mu_0\delta_0,\\
\\
|\nabla v_n|^2\mkern-10mu  &\rightharpoonup d\overline{\mu}\ge |\nabla \tilde{v}|^2+\sum_{k\in\mathfrak{K}}\overline{\mu}_k\delta_{y_k}+\overline{\mu}_0\delta_0,\\
\\
|u_n|^{2^*} \mkern-15mu &\rightharpoonup d\rho= |\tilde{u}|^{2^*}+\sum_{j\in\mathfrak{J}}\rho_j\delta_{x_j}+\rho_0\delta_0,\\
\\
|v_n|^{2^*} \mkern-15mu  &\rightharpoonup d\overline{\rho}= |\tilde{v}|^{2^*}+\sum_{k\in\mathfrak{K}}\overline{\rho}_k\delta_{y_k}+\overline{\rho}_0\delta_0,\\
\\
\dfrac{u_n^2}{|x|^2} \mkern-10mu &\rightharpoonup d\gamma=\dfrac{\tilde{u}^2}{|x|^2}+\gamma_0\delta_0,\\
\\
\dfrac{v_n^2}{|x|^2} \mkern-10mu &\rightharpoonup d\overline{\gamma}=\dfrac{\tilde{v}^2}{|x|^2}+\overline{\gamma}_0\delta_0.
\end{array}
\right.
\end{equation}
Note that, thanks to the Sobolev and Hardy inequalities (\eqref{sobolev_inequality} and \eqref{hardy_inequality} resp.) the above numbers satisfy the inequalities
\begin{equation}\label{ineq:sobcon}
\begin{split}
\mathcal{S}\rho_j^{\frac{2}{2^*}}&\leq\mu_j\qquad\text{for all }j\in\mathfrak{J}\cup\{0\},\\
\mathcal{S}\overline{\rho}_k^{\frac{2}{2^*}}&\leq\overline{\mu}_k\qquad\text{for all }k\in\mathfrak{K}\cup\{0\},
\end{split}
\end{equation}
and
\begin{equation}\label{ineq:harcon}
\begin{split}
\Lambda_N\gamma_0&\leq\mu_0,\\
\Lambda_N\overline{\gamma}_0&\leq\overline{\mu}_0.
\end{split}
\end{equation}
The concentration at infinity of the sequence $\{u_n\}$ is encoded by the numbers
\begin{equation}\label{con-compinfty}
\begin{split}
\mu_{\infty}&=\lim\limits_{R\to+\infty}\limsup\limits_{n\to+\infty}\int_{|x|>R}|\nabla u_n|^{2}dx,\\
\rho_{\infty}&=\lim\limits_{R\to+\infty}\limsup\limits_{n\to+\infty}\int_{|x|>R}|u_n|^{2^*}dx,\\
\gamma_{\infty}&=\lim\limits_{R\to+\infty}\limsup\limits_{n\to+\infty}\int_{|x|>R}\frac{u_n^{2}}{|x|^2}dx.
\end{split}
\end{equation}
The concentration at infinity of the sequence $\{v_n\}$ is encoded by the numbers $\overline{\mu}_{\infty}$, $\overline{\rho}_{\infty}$ and $\overline{\gamma}_{\infty}$ defined analogously.\newline
Next, for $j\in\mathfrak{J}$, let $\varphi_{j,\varepsilon}(x)$ be a smooth cut-off function centered at $x_j$, i.e., $\varphi_{j,\varepsilon}\in C^{\infty}(\mathbb{R}^+_0)$ and
\begin{equation}\label{cutoff}
\varphi_{j,\varepsilon}=1 \quad \hbox{in}\quad B_{\frac{\varepsilon}{2}}(x_j),\quad \varphi_{j,\varepsilon}=0 \quad \hbox{in}
\quad B_{\varepsilon}^c(x_j)\quad \hbox{and}\quad\displaystyle|\nabla \varphi_{j,\varepsilon}|\leq \frac{4}{\varepsilon},
\end{equation}
where $B_r(x_j)$ is the ball of radius $r>0$ centered at the point $x_j\in\mathbb{R}^N$. Thus, testing $\mathcal{J}_{\nu}'(u_n,v_n)$ with $(u_n\varphi_{j,\varepsilon},0)$ we get
\begin{equation}\label{eq:concompact1}
\begin{split}
0&=\lim\limits_{n\to+\infty}\left\langle \mathcal{J}_{\nu}'(u_n,v_n)\big|(u_n\varphi_{j,\varepsilon},0)\right\rangle\\
&=\lim\limits_{n\to+\infty}\left(\int_{\mathbb{R}^N}|\nabla u_n|^2\varphi_{j,\varepsilon}dx+\int_{\mathbb{R}^N}u_n\nabla u_n\nabla\varphi_{j,\varepsilon}dx-\lambda_1\int_{\mathbb{R}^N}\frac{u_n^2}{|x|^2}\varphi_{j,\varepsilon}dx\right.\\
&\mkern+80mu-\left.\int_{\mathbb{R}^N}|u_n|^{2^*}\varphi_{j,\varepsilon}dx-\nu\alpha\int_{\mathbb{R}^N}h(x)|u_n|^\alpha  |v_n|^\beta  \varphi_{j,\varepsilon} dx\right)\\
&=\int_{\mathbb{R}^N}\varphi_{j,\varepsilon}d\mu+\int_{\mathbb{R}^N}u_0\nabla u_0\nabla\varphi_{j,\varepsilon}dx-\lambda_1\int_{\mathbb{R}^N}\varphi_{j,\varepsilon}d\gamma\\
&\mkern+25mu-\int_{\mathbb{R}^N}\varphi_{j,\varepsilon}d\rho-\nu\alpha\int_{\mathbb{R}^N}h(x)|u_0|^\alpha |v_0|^{\beta} \varphi_{j,\varepsilon} \, dx.
\end{split}
\end{equation}
Taking $\varepsilon>0$ small enough, then $0\notin supp(\varphi_{j,\varepsilon})$. Thus, since $h\in L^{\infty}(\mathbb{R}^N)$, letting $\varepsilon\to0$ we get $\mu_j-\rho_j\leq0$ from the expression \eqref{eq:concompact1}. Then, we have two options either the PS sequence has a convergent subsequence or it concentrates around some of the points $x_j$.\newline
In other words,
\begin{equation}\label{afr}
\text{Either }\rho_j=0\ \forall j\in\mathfrak{J}\quad \text{ or, by \eqref{ineq:sobcon},}\quad\rho_j\ge \mathcal{S}^{\frac{N}{2}}\quad\forall j\in\mathfrak{J},\text{ so that }\mathfrak{J}\text{ is finite}.
\end{equation}
By a similar argument the same conclusion holds for the numbers $\overline{\rho}_k$, i.e.,
\begin{equation}\label{afr2}
\text{Either }\overline{\rho}_k=0\ \forall k\in\mathfrak{K}\quad \text{ or, by \eqref{ineq:sobcon},}\quad\overline{\rho}_k\ge \mathcal{S}^{\frac{N}{2}}\quad\forall k\in\mathfrak{K},\text{ so that }\mathfrak{K}\text{ is finite}.
\end{equation}
As we defined $\varphi_{j,\varepsilon}$ in \eqref{cutoff}, consider now $\varphi_{0,\varepsilon}$ a smooth cut-off function centered at the origin. Analogously, testing $\mathcal{J}_{\nu}'(u_n,v_n)$ with $(u_n\varphi_{0,\varepsilon},0)$ , we obtain $\mu_0-\lambda_1\gamma_0-\rho_0\leq 0$ and $\overline{\mu}_0-\lambda_2\overline{\gamma}_0-\overline{\rho}_0\leq0$. Taking in mind \eqref{ineq:harcon} and the definition of the constant $\mathcal{S}(\lambda)$ given in \eqref{Slambda}, we find
\begin{equation}\label{ineq:con0_a}
\begin{split}
\mu_0-\lambda_1\gamma_0&\ge \mathcal{S}(\lambda_1)\rho_0^{\frac{2}{2^*}},\\
\overline{\mu}_0-\lambda_2\overline{\gamma}_0&\ge \mathcal{S}(\lambda_2)\overline{\rho}_0^{\frac{2}{2^*}},
\end{split}
\end{equation}
from where we conclude
\begin{equation}\label{ineq:con0}
\begin{split}
\rho_0&=0\quad\text{or}\quad \rho_0\ge \mathcal{S}^{\frac{N}{2}}(\lambda_1),\\
\overline{\rho}_0&=0\quad\text{or}\quad \overline{\rho}_0\ge \mathcal{S}^{\frac{N}{2}}(\lambda_2).
\end{split}
\end{equation}
Finally, for $R>0$ sufficiently large such that $\{x_j\}_{j\in\mathfrak{J}}\cup \{0\} \subset  B_{R}(0)$, consider $\varphi_{\infty,\varepsilon}$ a cut-off function supported near $\infty$, i.e.,
\begin{equation}\label{cutoffinfi}
\varphi_{\infty,\varepsilon}=0 \quad \hbox{in}\quad B_{R}(0),\quad \varphi_{\infty,\varepsilon}=1 \quad \hbox{in}
\quad B_{R+1}^c(0)\quad \hbox{and}\quad\displaystyle|\nabla \varphi_{\infty,\varepsilon}|\leq \frac{4}{\varepsilon}.
\end{equation}

Testing $\mathcal{J}_{\nu}'(u_n,v_n)$ with $(u_n\varphi_{\infty,\varepsilon},0)$, we can similarly prove that $\mu_{\infty}-\lambda_1\gamma_{\infty}-\rho_{\infty}\leq 0$ as well as $\overline{\mu}_{\infty}-\lambda_2\overline{\gamma}_{\infty}-\overline{\rho}_{\infty}\leq0$ and, as above, we therefore find
\begin{equation}\label{ineq:coninf_a}
\begin{split}
\mu_{\infty}-\lambda_1\gamma_{\infty}&\ge \mathcal{S}(\lambda_1)\rho_{\infty}^{\frac{2}{2^*}},\\
\overline{\mu}_{\infty}-\lambda_2\overline{\gamma}_{\infty}&\ge \mathcal{S}(\lambda_2)\overline{\rho}_{\infty}^{\frac{2}{2^*}},
\end{split}
\end{equation}
and we also conclude
\begin{equation}\label{ineq:coninf}
\begin{split}
\rho_{\infty}&=0\quad\text{or}\quad \rho_{\infty}\ge \mathcal{S}^{\frac{N}{2}}(\lambda_1),\\
\overline{\rho}_{\infty}&=0\quad\text{or}\quad \overline{\rho}_{\infty}\ge \mathcal{S}^{\frac{N}{2}}(\lambda_2).
\end{split}
\end{equation}
Next, from \eqref{eq:limit2} we get
\begin{equation*}
c=\left(\frac{1}{2}-\frac{1}{\alpha+\beta} \right)\|(u_n,v_n)\|_{\mathbb{D}}^2+\left(\frac{1}{\alpha+\beta}- \frac{1}{2^*} \right)\left( \|u_n\|_{L^{2^*}}^{2^*}+\|v_n\|_{L^{2^*}}^{2^*} \right)+o(1).
\end{equation*}
Hence, because of \eqref{con-comp}, \eqref{ineq:sobcon}, \eqref{ineq:harcon}, \eqref{ineq:con0_a} and \eqref{ineq:coninf_a} above, we find
\begin{equation}\label{ineq:larga}
\begin{split}
c\ge& \left(\frac{1}{2}-\frac{1}{\alpha+\beta} \right)\Bigg(\|(\tilde{u},\tilde{v})\|_{\mathbb{D}}^2+ \sum_{j\in\mathfrak{J}}\mu_j + (\mu_0-\lambda_1\gamma_0)+(\mu_{\infty}-\lambda_1\gamma_{\infty}) \\
&\mkern+105mu +\left. \sum_{k\in\mathfrak{K}} \overline{\mu}_k+ (\overline{\mu}_0-\lambda_2\overline{\gamma}_0)+(\overline{\mu}_{\infty}-\lambda_2\overline{\gamma}_{\infty})\right)\\
&+\left(\frac{1}{\alpha+\beta}- \frac{1}{2^*} \right)\left(\int_{\mathbb{R}^N}|\tilde{u}|^{2^*}dx+\int_{\mathbb{R}^N}|\tilde{v}|^{2^*}dx \right. \\
& \mkern+120mu +  \sum_{j\in\mathfrak{J}}\rho_j+\rho_0+\rho_{\infty}+\left. \sum_{k\in\mathfrak{K}}\overline{\rho}_k+\overline{\rho}_0+\overline{\rho}_{\infty} \right) \\
\ge&\left(\frac{1}{2}-\frac{1}{\alpha+\beta} \right)\left( \mathcal{S}\left[\sum_{j\in\mathfrak{J}}\rho_j^{\frac{2}{2^*}}+\sum_{k\in\mathfrak{K}}\overline{\rho}_k^{\frac{2}{2^*}}\right]+\mathcal{S}(\lambda_1)\left[\rho_0^{\frac{2}{2^*}}+\rho_{\infty}^{\frac{2}{2^*}}\right]+\mathcal{S}(\lambda_2)\left[\overline{\rho}_0^{\frac{2}{2^*}}+\overline{\rho}_{\infty}^{\frac{2}{2^*}}\right]\right)\\
&+\left(\frac{1}{\alpha+\beta}- \frac{1}{2^*} \right)\left(\sum_{j\in\mathfrak{J}}\rho_j+\rho_0+\rho_{\infty}+\sum_{k\in\mathfrak{K}}\overline{\rho}_k+\overline{\rho}_0+\overline{\rho}_{\infty} \right).
\end{split}
\end{equation}
In case of having concentration at the point $x_j$, i.e., $\rho_j>0$,  from above and \eqref{afr} we find
\begin{equation*}
c\ge\left(\frac{1}{2}-\frac{1}{\alpha+\beta} \right)\mathcal{S}^{1+\frac{N}{2}\frac{2}{2^*}}+\left(\frac{1}{\alpha+\beta}- \frac{1}{2^*} \right)\mathcal{S}^{\frac{N}{2}}=\frac{1}{N}\mathcal{S}^{\frac{N}{2}},
\end{equation*}
and we reach a contradiction with the hypothesis on the energy level $c$  \eqref{hyplemmaPS2}. Then, we conclude $\rho_j=\mu_j=0$ for all $j\in\mathfrak{J}$. Arguing similarly we also conclude that $\overline{\rho}_k=\overline{\mu}_k=0$ for all $k\in\mathfrak{K}$.\newline
If $\rho_0\neq0$ from the above chain of inequalities and \eqref{ineq:con0} we get
\begin{equation*}
c\ge\frac{1}{N}\mathcal{S}^{\frac{N}{2}}(\lambda_1),
\end{equation*}
and we also reach a contradiction with the hypothesis on the energy level $c$. Then, $\rho_0=0$. Similarly, we also get $\overline{\rho}_0=0$. Arguing as above and using \eqref{ineq:coninf} we also find $\rho_{\infty}=0$ and $\overline{\rho}_{\infty}=0$. Hence, there exists a subsequence that strongly converges in $L^{2^*}(\mathbb{R}^N)\times L^{2^*}(\mathbb{R}^N)$. As a consequence, we have
\begin{equation*}
\|(u_n-\tilde{u},v_n-\tilde{v})\|_{\mathbb{D}}^2 = \left\langle \mathcal{J}_\nu'(u_n,v_n)\big| (u_n-\tilde{u},v_n-\tilde{v}) \right\rangle + o(1),
\end{equation*}
and, thus, the sequence $\{(u_n,v_n)\}$ strongly converges in $\mathbb{D}$ and the PS condition holds.

\end{proof}

The next result improves Lemma~\ref{lemmaPS2} by extending the PS condition to supercritical energy levels, excluding multiples or combinations of the critical ones. In fact, this is a generalization of \cite[Lemma 3.5]{AbFePe} for system \eqref{system:alphabeta}. Whereas in the aforementioned result it is assumed that $\min\{\alpha,\beta\}\ge 2$, we weaken the hypotheses to require that $\max\{\alpha,\beta\}\ge 2$. The main idea of the proof is that if a sequence of solutions to \eqref{system:alphabeta} does not strongly converge in $\mathbb{D}$, then we can quantify the energy levels. As a consequence, we obtain the admissible energy levels which allow us to pass to the limit. In order to do it, we shall apply Lemma~\ref{algelemma} into the component whose exponent is bigger than or equal to $2$.

\

To study positive solutions of \eqref{system:alphabeta}, it will be useful to consider the modified problem
\beq\label{systemp}
\left\{\begin{array}{ll}
\displaystyle -\Delta u - \lambda_1 \frac{u}{|x|^2}-(u^+)^{2^*-1}=  \nu\alpha h(x) (u^+)^{\alpha-1}\, (v^+)^{\beta}  &\text{in }\mathbb{R}^N,\vspace{.3cm}\\
\displaystyle -\Delta v - \lambda_2 \frac{v}{|x|^2}-(v^+)^{2^*-1}= \nu \beta h(x)  (u^+)^{\alpha}\, (v^+)^{\beta-1} &\text{in }\mathbb{R}^N,
\end{array}\right.
\eeq
where  $u^+=\max\{u,0\}$. We also denote the negative part of the function $u$ as $u^-=\min\{u,0\}$, so that $u=u^+ + u^-$. Observe that a solution $(u,v)$ of \eqref{system:alphabeta} satisfies \eqref{systemp}.

The above system admits a variational structure and its solutions correspond to critical points of the energy functional
\begin{equation}\label{funct:SKdVp}
\mathcal{J}^+_\nu (u,v)=\|(u,v)\|^2_{\mathbb{D}}
- \frac{1}{2^*} \left( \|u^+ \|_{L^{2^*}}^{2^*}+\|v^+ \|_{L^{2^*}}^{2^*}  \right) -\nu \int_{\mathbb{R}^N} h(x) (u^+)^{\alpha}\, (v^+)^{\beta} \, dx
\end{equation}
defined in $\mathbb{D}$. Besides, we shall denote $\mathcal{N}^+_\nu$ the Nehari manifold associated to $\mathcal{J}^+_\nu $. More precisely,
\begin{equation*}
\mathcal{N}^+_\nu=\left\{ (u,v) \in \mathbb{D} \setminus (0,0) \, : \,  \left\langle (\mathcal{J}_\nu^+)' (u,v){\big|}(u,v)\right\rangle=0 \right\}.
\end{equation*}

Given $(u,v) \in \mathcal{N}^+_\nu$, the following identity is satisfied
\begin{equation} \label{Nnueqp}
\|(u,v)\|_\mathbb{D}^2=\left( \|u^+ \|_{L^{2^*}}^{2^*}+\|v^+ \|_{L^{2^*}}^{2^*}  \right) +\nu (\alpha+\beta) \int_{\mathbb{R}^N} h(x) (u^+)^{\alpha} (v^+)^{\beta} \, dx.
\end{equation}

Let us also note that, on the Nehari manifold $\mathcal{N}^+_\nu$, the functional $\mathcal{J}^+_{\nu}$ can be written as
\beq\label{Nnueqp1}
\mathcal{J}^+_{\nu}{\big|}_{\mathcal{N}_\nu} (u,v) = \left( \frac{1}{2}-\frac{1}{\alpha+\beta} \right) \|(u,v)\|^2_{\mathbb{D}} +  \left( \frac{1}{\alpha+\beta}-\frac{1}{2^*} \right) \left( \|u^+ \|_{L^{2^*}}^{2^*}+\|v^+ \|_{L^{2^*}}^{2^*}  \right).
\eeq

\

\begin{lemma}\label{lemmaPS1}
Assume $\alpha+\beta<2^*$ and \eqref{H}, $\alpha\ge2$, $\lambda_2 \ge \lambda_1$ and
\begin{equation}\label{PS0}
\mathcal{S}^{\frac{N}{2}}(\lambda_1)+\mathcal{S}^{\frac{N}{2}}(\lambda_2)<\mathcal{S}^{\frac{N}{2}}.
\end{equation}
Then, there exists $\tilde{\nu}>0$ such that, if $0<\nu\leq\tilde{\nu}$ and $\{(u_n,v_n)\} \subset \mathbb{D}$ is a PS sequence for $\mathcal{J}^+_\nu$ at level $c\in\mathbb{R}$ such that
\beq\label{PS1}
\frac{1}{N} \mathcal{S}^{\frac{N}{2}}(\lambda_2)<c<\frac{1}{N} \left(\mathcal{S}^{\frac{N}{2}}(\lambda_1)+\mathcal{S}^{\frac{N}{2}}(\lambda_2) \right),
\eeq
and
\beq\label{PS2}
c\neq \frac{\ell}{N} \mathcal{S}^{\frac{N}{2}}(\lambda_2) \quad \mbox{ for every } \ell \in \mathbb{N}\setminus \{0\},
\eeq
then $(u_n,v_n)\to(\tilde{u},\tilde{v}) \in \mathbb{D}$ up to subsequence.
\end{lemma}

\begin{proof}

Analogously to Lemma~\ref{lemmaPS0}, any PS sequence for $\mathcal{J}_\nu^+$ is bounded in $\mathbb{D}$. Consequently,  there exists a subsequence $\{(u_n,v_n)\}$ which weakly converges to $(\tilde{u},\tilde{v}) \in \mathbb{D}$. Since $(\mathcal{J}^+_\nu)'(u_n,v_n)\to 0$ in $\mathbb{D}'$, then
\begin{equation*}
\left\langle (\mathcal{J}^+_\nu)'(u_n,v_n){\big|} (u_n^-,0)\right\rangle= \int_{\mathbb{R}^N} |\nabla u_n^-|^2 \, dx - \lambda_1 \int_{\mathbb{R}^N} \dfrac{(u_n^-)^2}{|x|^2} \, dx  \to0,
\end{equation*}
which implies that $u_n^-\to 0$ strongly in $\mathcal{D}^{1,2} (\mathbb{R}^N)$. Similarly, we can also prove $v_n^-\to 0$.
Therefore, we can consider $\{(u_n,v_n)\}$ as a non-negative PS sequence at level $c$ for the energy functional $\mathcal{J}_\nu$.\newline
Next, following closely the approach of the proof of Lemma~\ref{lemmaPS2}, we can deduce the existence of a subsequence, still denoted by $\{(u_n,v_n)\}$, two (at most countable) sets of points $\{x_j\}_{j\in\mathfrak{J}}\subset\mathbb{R}^N$ and $\{y_k\}_{k\in\mathfrak{K}}\subset\mathbb{R}^N$, and positive numbers $\{\mu_j,\rho_j\}_{j\in\mathfrak{J}}$, $\{\overline{\mu}_k,\overline{\rho}_k\}_{k\in\mathfrak{K}}$, $\mu_0$, $\rho_0$, $\gamma_0$, $\overline{\mu}_0$, $\overline{\rho}_0$ and $\overline{\gamma}_0$ such that the weak convergence given by \eqref{con-comp} is satisfied. Besides, the inequalities \eqref{afr}, \eqref{afr2}, \eqref{ineq:con0_a}, \eqref{ineq:con0} hold.\newline
In a similar way, we define the concentration at infinity provided by the values $\mu_\infty$, $\rho_\infty$, $\overline{\mu}_\infty$ and $\overline{\rho}_\infty$ as in \eqref{con-compinfty}, for which \eqref{ineq:coninf_a} and \eqref{ineq:coninf} hold.

\

Now, we claim that:
\beq\label{claimPS21}
\mbox{ either } u_n\to \tilde u\  \mbox{  strongly in } L^{2^*}(\mathbb{R}^N) \qquad \mbox{or} \qquad v_n\to \tilde v\ \mbox{  strongly in }L^{2^*}(\mathbb{R}^N).
\eeq

Assume by contradiction that $\{u_n\}$ and $\{v_n\}$ do not strongly converge in $L^{2^*}(\mathbb{R}^N)$. This implies that there exists $j\in \mathfrak{J}\cup \{0 \cup \infty\}$ and $k\in \mathfrak{J}\cup \{0 \cup \infty\}$ such that $\rho_{j}>0$ and $\overline{\rho}_k>0$. Finally, by the expressions \eqref{eq:limit2}, \eqref{afr}, \eqref{afr2}, \eqref{ineq:con0_a} and \eqref{ineq:con0} applied in \eqref{ineq:larga}, we get
\begin{equation*}
\begin{split}
c=&\left( \frac{1}{2} -\frac{1}{\alpha+\beta} \right)\|(u_n,v_n)\|_{\mathbb{D}}^2+\left(\frac{1}{\alpha+\beta}- \frac{1}{2^*} \right)\left( \|u_n \|_{L^{2^*}}^{2^*}+\|v_n \|_{L^{2^*}}^{2^*}  \right)+o(1)\\
\ge&\left( \frac{1}{2} -\frac{1}{\alpha+\beta} \right)\left(\mathcal{S}(\lambda_1)\rho_j^{\frac{2}{2^*}} +\mathcal{S}(\lambda_2) \overline{\rho}_k^{\frac{2}{2^*}}\right)+\left(\frac{1}{\alpha+\beta}- \frac{1}{2^*} \right)\left(\rho_j+\overline{\rho}_k \right)\\
\ge& \ \frac{1}{N} \left(\mathcal{S}^{\frac{N}{2}}(\lambda_1) + \mathcal{S}^{\frac{N}{2}}(\lambda_2)  \right) .
\end{split}
\end{equation*}
The previous inequality contradicts assumption \eqref{PS1}, so claim \eqref{claimPS21} is proved.

\

Subsequently, we claim that:
\beq\label{claimPS22}
\mbox{ either } u_n\to  \tilde u\  \mbox{ in } \mathcal{D}^{1,2}(\mathbb{R}^N) \qquad \mbox{ or } \qquad v_n\to  \tilde v \ \mbox{ in } \mathcal{D}^{1,2}(\mathbb{R}^N).
\eeq

Without loss of generality, by \eqref{claimPS21}, assume that the sequence $\{u_n\}$ strongly converges in $L^{2^*}(\mathbb{R}^N)$. Then, it is enough to observe that
\begin{equation*}
\|u_n-\tilde u\|_{\lambda_1}^2 = \left\langle \mathcal{J}_\nu'(u_n,v_n){\big|} (u_n-\tilde u,0) \right\rangle + o(1).
\end{equation*}
This implies that $u_n\to \tilde u$ in $\mathcal{D}^{1,2}(\mathbb{R}^N)$. Repeating the argument for the sequence $\{v_n\}$ we conclude \eqref{claimPS22}.
\

Next, we show that both components  strongly converge in $\mathcal{D}^{1,2}(\mathbb{R}^N)$. In order to do it, we shall distinguish the proof between two cases:

\

\textbf{Case 1}: The sequence $\{v_n\}$ strongly converges to $\tilde{v}$ in $\mathcal{D}^{1,2}(\mathbb{R}^N)$.

\

We want to prove that $\{u_n\}$ strongly converges to $\tilde{u}$ in $\mathcal{D}^{1,2}(\mathbb{R}^N)$. By contradiction, suppose that none of its subsequences converge. If one assumes that $\mathfrak{J}\cup \{0,\infty\}$ contains more than one point then, by combining \eqref{ineq:larga} with \eqref{afr}, \eqref{ineq:con0_a}, \eqref{ineq:con0}, \eqref{ineq:coninf_a} and \eqref{ineq:coninf}, we find
\begin{equation*}
c\ge \frac{2}{N} \mathcal{S}^{\frac{N}{2}}(\lambda_1)\ge \frac{1}{N} \left( \mathcal{S}^{\frac{N}{2}}(\lambda_1) + \mathcal{S}^{\frac{N}{2}}(\lambda_2) \right)
\end{equation*}
because of $\lambda_2\ge \lambda_1$ and $\mathcal{S}(\lambda)$ is decreasing, which contradicts assumption \eqref{PS1}. So, assume that there exists only one concentration point $x_j$, $j\in \mathfrak{J}\cup \{0,\infty\}$, for the sequence $\{u_n\}$.

Next, we shall prove that $\tilde{v}\not \equiv 0$. Assume by contradiction that $\tilde{v}\equiv 0$, then $\tilde{u}\ge 0 $ and $\tilde{u}$ verifies
\beq\label{utildeq}
-\Delta \tilde{u} - \lambda_1 \frac{\tilde{u}}{|x|^2}=\tilde{u}^{2^*-1} \qquad \mbox{ in } \mathbb{R}^N.
\eeq
Thus, $\tilde{u}=z_\mu^{\lambda_1}$ for some $\mu>0$ and $\displaystyle\int_{\R^N} \tilde{u}^{2^*} \, dx =\mathcal{S}^{\frac{N}{2}}(\lambda_1)$ by \eqref{normcrit}. Since $\{u_n\}$ concentrates at one point, by combining \eqref{ineq:larga} with \eqref{afr}, \eqref{ineq:con0_a}, \eqref{ineq:con0}, we conclude that
$$
c\ge \frac{1}{N}\left( \int_{\R^N} \tilde{u}^{2^*} \, dx +\mathcal{S}^{\frac{N}{2}}(\lambda_1)  \right) = \frac{2}{N}\mathcal{S}^{\frac{N}{2}}(\lambda_1) \ge \frac{1}{N}\left( \mathcal{S}^{\frac{N}{2}}(\lambda_1)+ \mathcal{S}^{\frac{N}{2}}(\lambda_2) \right),
$$
which contradicts the assumption \eqref{PS1}.

If $\tilde{v}\equiv 0$ and $\tilde{u}\equiv 0$, one has that $u_n$ should verify the following
$$
-\Delta u_n - \lambda_2 \frac{u_n}{|x|^2}-u_n^{2^*-1}=o(1) \qquad \mbox{ in the dual space } \left( \mathcal{D}^{1,2} (\mathbb{R}^N)\right)',
$$
and
$$
c=\mathcal{J}_\nu(u_n,v_n)+o(1)=\frac{1}{N} \int_{\R^N} u_n^{2^*}+o(1)\to \frac{1}{N} \rho_j,
$$
as $\{u_n\}$ concentrates at one point. If $j\in \mathfrak{J}$, then $\{u_n\}$ is a positive PS sequence for the functional
$$
\mathcal{J}_j(u)=\frac{1}{2} \int_{\R^N} |\nabla u|^2\, dx -\frac{1}{2^*} \int_{\R^N} |u|^{2^*} \, dx.
$$
Using the characterization of PS sequences for the functional $\mathcal{J}_j$ provided by \cite{Stru}, we determine that $\rho_j=\ell \mathcal{S}^{\frac{N}{2}}$ for some $\ell\in\mathbb{N}$, which contradicts the hypotheses \eqref{PS0} and \eqref{PS1}. We conclude that $\mathfrak{J}=\emptyset$. In case that $\{u_n\}$ concentrates at zero or infinity, we can follow an analogous approach for the energy functional $\mathcal{J}_1$, defined in \eqref{funct:Ji}, and the result provided by \cite{Smets} to obtain that
$$
c=\mathcal{J}_\nu(u_n,v_n)+o(1)=\mathcal{J}_1(u_n)+o(1)\to \frac{\ell}{N} \mathcal{S}^{\frac{N}{2}}(\lambda_1),
$$
with $\ell \in \mathbb{N}\cup \{0\}$ contradicting to \eqref{PS1}. As a consequence, we have proved $\tilde{v}\gneq 0$ in $\mathbb{R}^N$. Next, we shall prove that $u_n\weakto\tilde{u}$ in $D^{1,2}(\R^N)$ such that $\tilde{u}\not \equiv 0$. As before, reasoning by contradiction assume that $\tilde{u}=0$, then $\tilde{v}$ satisfies the problem
\beq\label{vtildeq}
-\Delta \tilde{v} - \lambda_2 \frac{\tilde{v}}{|x|^2}=\tilde{v}^{2^*-1} \qquad \mbox{ in } \mathbb{R}^N.
\eeq
Thus, $\tilde{v}=z_\mu^{\lambda_2}$ for some $\mu>0$ and $\displaystyle\int_{\R^N} \tilde{v}^{2^*} \, dx =\mathcal{S}^{\frac{N}{2}}(\lambda_2)$ by \eqref{normcrit}. As a consequence, combining \eqref{ineq:larga} with \eqref{afr}, \eqref{ineq:con0_a}, \eqref{ineq:con0}, we get
$$
c\ge \frac{1}{N}\left( \int_{\R^N} \tilde{v}^{2^*} \, dx +\mathcal{S}^{\frac{N}{2}}(\lambda_1)  \right) = \frac{1}{N}\left( \mathcal{S}^{\frac{N}{2}}(\lambda_1)+ \mathcal{S}^{\frac{N}{2}}(\lambda_2) \right),
$$
contradicting \eqref{PS1}. Thus, $\tilde{u},\tilde{v} \not \equiv 0$. Next, taking the equality
\begin{equation}\label{eqlemmaPS10}
\begin{split}
c&=\mathcal{J}_\nu(u_n,v_n)-\frac{1}{2}\left\langle \mathcal{J}_\nu'(u_n,v_n){\big|} (u_n,v_n) \right\rangle + o(1)\\
& = \frac{1}{N}\left( \|u_n \|_{L^{2^*}}^{2^*}+\|v_n \|_{L^{2^*}}^{2^*}  \right)+ \nu \frac{\alpha+\beta-2}{2} \int_{\R^N} h(x) u_n^\alpha v_n^\beta \, dx + o(1)\\
& \to \frac{1}{N} \left( \|\tilde u \|_{L^{2^*}}^{2^*}+\|\tilde v \|_{L^{2^*}}^{2^*}  \right)  + \frac{\rho_j}{N} +   \nu \frac{\alpha+\beta-2}{2} \int_{\R^N} h(x) \tilde{u}^\alpha \tilde{v}^{\beta} \, dx \quad\text{as } n\to+\infty,
\end{split}
\end{equation}
by the concentration at $j \in \mathfrak{J}\cup \{ 0, \infty\}$.

Since $\left\langle \mathcal{J}_\nu'(u_n,v_n){\big|} (\tilde{u},\tilde{v}) \right\rangle \to 0$, ones arrives at the expression
$$
\|(\tilde{u},\tilde{v}) \|_{\mathbb{D}}=  \|\tilde u \|_{L^{2^*}}^{2^*}+\|\tilde v \|_{L^{2^*}}^{2^*}  + \nu(\alpha+\beta) \int_{\R^N} h(x) \tilde{u}^\alpha \tilde{v}^\beta \, dx,
$$
which is equivalent to say that $(\tilde{u},\tilde{v}) \in\mathcal{N}_\nu$. Actually, using \eqref{eqlemmaPS10}, \eqref{eqlemmaPS11}, \eqref{Nnueq}, \eqref{afr}, \eqref{ineq:con0_a}, \eqref{ineq:con0} and  \eqref{ineq:larga}, we have that
\begin{equation*}
\begin{split}
\mathcal{J}_\nu(\tilde{u},\tilde{v})&=\frac{1}{N}\left( \|\tilde u \|_{L^{2^*}}^{2^*}+\|\tilde v  \|_{L^{2^*}}^{2^*}  \right) + \nu \frac{\alpha+\beta-2}{2} \int_{\R^N} h(x) \tilde{u}^\alpha \tilde{v}^\beta \, dx \\
&=c - \frac{\rho_j}{N}\\
&<\frac{1}{N}\left( \mathcal{S}^{\frac{N}{2}}(\lambda_1)+ \mathcal{S}^{\frac{N}{2}}(\lambda_2) \right) - \frac{1}{N} \mathcal{S}^{\frac{N}{2}}(\lambda_1)\\
&= \frac{1}{N} \mathcal{S}^{\frac{N}{2}}(\lambda_2).
\end{split}
\end{equation*}
The above expression implies that
$$
\tilde{c}_\nu= \inf_{(u,v)\in\mathcal{N}_\nu} \mathcal{J}_\nu(u,v) < \frac{1}{N} \mathcal{S}^{\frac{N}{2}}(\lambda_2).
$$
However, for $\nu$ sufficiently small, Theorem~\ref{thm:groundstatesalphabeta} states that $\tilde c_\nu =  \frac{1}{N} \mathcal{S}^{\frac{N}{2}}(\lambda_2)$, which contradicts the former inequality. Thus, we have proved that $u_n\to \tilde{u}$ strongly in $\mathcal{D}^{1,2}(\R^N)$.

\

\textbf{Case 2}: The sequence $\{u_n\}$ strongly converges to $\tilde{u}$ in $\mathcal{D}^{1,2}(\mathbb{R}^N)$.

\

We want to prove that $\{v_n\}$ strongly converges to $\tilde{v}$ in $\mathcal{D}^{1,2}(\mathbb{R}^N)$. By contradiction, suppose that none of its subsequences converge. We start by proving that $\tilde{u}\not \equiv 0$. Assuming that $\tilde{u} \equiv 0$ by contradiction, then $\{v_n\}$ is a PS sequence for the energy functional $\mathcal{J}_2$ defined in \eqref{funct:Ji} at energy level $c$.

Since $v_n\weakto\tilde{v}$ in $\mathcal{D}^{1,2}(\R^N)$, where $\tilde{v}$ satisfies the entire problem \eqref{vtildeq}, then, $\tilde{v}=z_\mu^{\lambda_2}$ for some $\mu>0$.  Moreover, applying the compactness theorem given by \cite{Smets} and \eqref{Jzeta}, one has
$$
c= \lim_{n\to+\infty} \mathcal{J}_2 (v_n) =  \mathcal{J}_2 (z_\mu^{\lambda_2})+ \frac{m}{N} \mathcal{S}^{\frac{N}{2}}+ \frac{\ell}{N} \mathcal{S}^{\frac{N}{2}}(\lambda_2)=\frac{m}{N} \mathcal{S}^{\frac{N}{2}}+ \frac{\ell+1}{N} \mathcal{S}^{\frac{N}{2}}(\lambda_2),
$$
where $m\in \mathbb{N}$ and $\ell \in \mathbb{N}\cup \{0\}$, in contradiction with \eqref{PS1} and \eqref{PS2}. So, we can conclude that $\tilde{u} \not \equiv 0$.

Conversely, if one assumes that $\tilde v \equiv 0$, then one has that $\tilde{u}$ solves to \eqref{utildeq}, which implies that $u=z_\mu^{\lambda_1}$ for some $\mu>0$. Subsequently, as we did above in \textbf{Case 1},
$$
c\ge \frac{1}{N}\left( \int_{\R^N} \tilde{u}^{2^*} \, dx +\mathcal{S}^{\frac{N}{2}}(\lambda_2)  \right) =   \frac{1}{N}\left( \mathcal{S}^{\frac{N}{2}}(\lambda_1)+ \mathcal{S}^{\frac{N}{2}}(\lambda_2) \right),
$$
contradicting \eqref{PS1}. Therefore, we can infer that $\tilde{u},\tilde{v}\not \equiv 0$. Since $(\tilde{u},\tilde{v})$ is a solution of \eqref{system:alphabeta}, one can deduce that
\begin{equation}\label{eqlemmaPS11}
 \mathcal{J}_\nu(\tilde{u},\tilde{v})=\frac{1}{N}\int_{\R^N} \left(  \tilde{u}^{2^*} + \tilde{v}^{2^*} \right) \, dx + \nu \frac{\alpha+\beta-2}{2} \int_{\R^N} h(x) \tilde{u}^\alpha \tilde{v}^\beta \, dx \leq c.
\end{equation}
Using again \eqref{eqlemmaPS10} and the assumption that $v_n$ does not strongly converge  in $\mathcal{D}^{1,2}(\R^N)$, then there exists at least $k\in\mathfrak{K}\cup \{0,\infty\}$ such that $\overline{\rho}_k>0$, which implies that
$$
c= \frac{1}{N} \left( \int_{\R^N}  \left(  \tilde{u}^{2^*} + \tilde{v}^{2^*} \right)  \, dx + \sum_{k\in\mathfrak{K}} \overline{\rho}_k+ \overline{\rho}_0 + \overline{\rho}_\infty \right)  + \nu \frac{\alpha+\beta-2}{2} \int_{\R^N} h(x) \tilde{u}^\alpha \tilde{v}^\beta \, dx .
$$
By using \eqref{eqlemmaPS11}, \eqref{afr2}, \eqref{ineq:con0_a}, \eqref{ineq:con0} and \eqref{PS1}, one gets
\beq\label{eqlemmaPS12}
\begin{split}
\mathcal{J}_\nu(\tilde{u},\tilde{v})&= c - \frac{1}{N} \sum_{k\in\mathfrak{K}} \overline{\rho}_k+ \overline{\rho}_0 + \overline{\rho}_\infty \\
& < \frac{1}{N}\left( \mathcal{S}^{\frac{N}{2}}(\lambda_1)+ \mathcal{S}^{\frac{N}{2}}(\lambda_2) \right) - \frac{1}{N} \mathcal{S}^{\frac{N}{2}}(\lambda_2)\\
& =  \frac{1}{N} \mathcal{S}^{\frac{N}{2}}(\lambda_1). \end{split}
\eeq
If we apply the definition of $\mathcal{S}^{\frac{N}{2}}(\lambda_1)$ and the first equation of \eqref{system:alphabeta}, we get
\beq\label{eqlemmaPS13}
\sigma_1 + \nu \int_{\R^N} h(x) \tilde{u}^\alpha \tilde{v}^\beta \, dx = \int_{\R^N} |\nabla \tilde{u}|^2 \, dx - \lambda_1 \int_{\R^N} \frac{\tilde{u}^2}{|x|^2} \, dx \ge \mathcal{S}(\lambda_1) \sigma^{2/2^*}_1,
\eeq
where $\displaystyle\sigma_1= \int_{\R^N} \tilde{u}^{2^*} \, dx$. By using Hölder's inequality, one gets
\beq\label{Holder}
\int_{\mathbb{R}^N} h(x)  \,  \tilde{u}^\alpha \tilde{v}^\beta \, dx \leq \|h\|_{L^{\infty}(\R^N)} \left( \int_{\mathbb{R}^N} \tilde{u}^{2^*}   \, dx \right)^{\frac{\alpha}{2^*}}\left( \int_{\mathbb{R}^N} \tilde{v}^{2^*}  \, dx   \right)^{\frac{\beta}{2^*}}.
\eeq
Combining \eqref{Holder} and \eqref{eqlemmaPS11}, we can transform \eqref{eqlemmaPS13} into
\beq\label{eqlemmaPS14}
\sigma_1 + C \nu \sigma_1^{\frac{\alpha}{2} \frac{N-2}{N}}\ge \mathcal{S}(\lambda_1) \sigma_1^\frac{N-2}{N}.
\eeq

Since $\tilde{v}\not \equiv 0$, there exits $\tilde \varepsilon>0$ such that $\displaystyle\int_{\R^N} \tilde{v}^{2^*} \, dx\ge \tilde \varepsilon$. Taking  $\varepsilon>0$ such that  $\tilde \varepsilon\ge \varepsilon \mathcal{S}^{\frac{N}{2}}(\lambda_1)$, by \eqref{eqlemmaPS14}, we can apply Lemma~\ref{algelemma} to get a fixed $\tilde \nu>0$ such that
$$
\sigma_1\ge (1-\varepsilon)\mathcal{S}^{\frac{N}{2}}(\lambda_1) \qquad \mbox{ for any } 0<\nu\leq \tilde{\nu}.
$$
From previous estimates and \eqref{eqlemmaPS11}, we obtain that
$$
\mathcal{J}_\nu(\tilde{u},\tilde{v}) \ge \frac{1}{N} \left((1-\varepsilon) \mathcal{S}^{\frac{N}{2}}(\lambda_1)+ \tilde{\varepsilon} \right)\ge  \frac{1}{N} \mathcal{S}^{\frac{N}{2}}(\lambda_1),
$$
which gives us a contradiction with \eqref{eqlemmaPS12}. Therefore, we can conclude that $v_n\to \tilde{v}$ strongly in $\mathcal{D}^{1,2}(\R^N)$.

\

In both cases, we have arrived at the conclusion that the PS sequences strongly converge in $\mathbb{D}$ to a non-vanishing limit, completing the proof.

\end{proof}

In a similar way, we can establish an analogous result in case that  $\lambda_1 \ge \lambda_2$.
\begin{lemma}\label{lemmaPS1a}
Assume $\alpha+\beta<2^*$ and \eqref{H}, $\beta\ge2$, $\lambda_1 \ge \lambda_2$ and
\begin{equation}\label{PS1a}
\mathcal{S}^{\frac{N}{2}}(\lambda_1)+\mathcal{S}^{\frac{N}{2}}(\lambda_2)<\mathcal{S}^{\frac{N}{2}}.
\end{equation}
Then, there exists $\tilde{\nu}>0$ such that, if $0<\nu\leq\tilde{\nu}$ and $\{(u_n,v_n)\} \subset \mathbb{D}$ is a PS sequence for $\mathcal{J}^+_\nu$ at level $c\in\mathbb{R}$ such that
\begin{equation}\label{PS2a}
\frac{1}{N} \mathcal{S}^{\frac{N}{2}}(\lambda_1)<c<\frac{1}{N} \left(\mathcal{S}^{\frac{N}{2}}(\lambda_1)+\mathcal{S}^{\frac{N}{2}}(\lambda_2) \right),
\end{equation}
and
\begin{equation*}
c\neq \frac{\ell}{N} \mathcal{S}^{\frac{N}{2}}(\lambda_1) \quad \mbox{ for every } \ell \in \mathbb{N}\setminus \{0\},
\end{equation*}
then $(u_n,v_n)\to(\tilde{u},\tilde{v}) \in \mathbb{D}$ up to subsequence.
\end{lemma}

\

\subsection{Critical range $\alpha+\beta=2^*$}\hfill\newline

To address the critical case, more hypotheses on the function $h$ are supposed. In particular,
\beq\label{hypH}\tag{H1}
h \mbox{ satisfies \eqref{H}} , \, h \mbox{ continuous around $0$ and $\infty$ and } h(0)=\lim_{x\to+\infty} h(x)=0.
\eeq
Moreover, it is distinguished the case in which either $h$ is radial or the case of a non-radial $h$, which requires the extra assumption that $\nu$ is sufficiently small.

In order to obtain minimizing and Mountain-Pass type solutions in the critical regime, namely Theorem~\ref{thm:nugrande}, Theorem~\ref{thm:lambdaground} and Theorem~\ref{MPgeom}, we make use of the following Lemma, which extends Lemma~\ref{lemmaPS2} and Lemma~\ref{lemmaPS1}. Let us emphasize that this result is an adaptation of \cite[Lemma 4.1]{AbFePe} to our setting.\newline
In what follows, $\mathbb{D}_r$ denotes the space of radially symmetric functions in $\mathbb{D}$.
\begin{lemma}\label{lemcritic}
Assume that  $\alpha+\beta=2^*$ and \eqref{hypH}. Let $\{(u_n,v_n)\} \subset \mathbb{D}_r$ be a PS sequence for $\mathcal{J}_\nu$ at level $c\in\mathbb{R}$ such that
\begin{itemize}
\item[i)] either $c$ satisfies \eqref{hyplemmaPS2}
\item[ii)] or $c$ satisfies \eqref{PS1} and \eqref{PS2} with $\alpha\ge 2$ and $\lambda_2\ge\lambda_1$,
\item[iii)] or $c$ satisfies \eqref{PS1a} and \eqref{PS2a} with $\beta\ge 2$ and $\lambda_1\ge\lambda_2$.
\end{itemize}
Then there exists $\tilde{\nu}>0$ such that for every $0<\nu\leq \tilde{\nu}$ then $(u_n,v_n)\to(\tilde{u},\tilde{v}) \in \mathbb{D}_r$ up to subsequence.
\end{lemma}

\begin{proof}
First, we note that, since $\{(u_n,v_n)\} \subset \mathbb{D}_r$ are radial functions, in case of having concentration at points different from $0$ or $\infty$ the set of concentration points is not a countable set, contradicting the \textit{concentration-compactness principle} by Lions (cf. \cite{Lions1,Lions2}).\newline
Then, arguing as in the proof of Lemma~\ref{lemmaPS2} and Lemma~\ref{lemmaPS1}, in order to avoid concentration at the origin, it is enough to prove (see \eqref{eq:concompact1}) that
\begin{equation}\label{radial:at0}
\lim\limits_{\varepsilon\to0}\limsup\limits_{n\to+\infty}\int_{\mathbb{R}^N}h(x)|u_n|^\alpha |v_n|^{\beta}\varphi_{0,\varepsilon}(x)dx=0,
\end{equation}
for $\varphi_{0,\varepsilon}$ a smooth cut-off function centered at 0 defined as in \eqref{cutoff}. Analogously, in order to avoid concentration at $\infty$, it suffices to show that
\begin{equation}\label{radial:atinf}
\lim\limits_{R\to+\infty}\limsup\limits_{n\to+\infty}\int_{|x|>R}h(x)|u_n|^\alpha |v_n|^{\beta}\varphi_{\infty,\varepsilon}(x)dx=0,
\end{equation}
where, $\varphi_{\infty,\varepsilon}$ is a cut-off function supported near $\infty$, introduced in \eqref{cutoffinfi}.

To prove \eqref{radial:at0}, we note that, by H\"older's inequality,
\begin{equation}\label{ineq:holder}
\begin{split}
\int_{\mathbb{R}^N}h(x)|u_n|^\alpha |v_n|^{\beta}\varphi_{0,\varepsilon}\,dx&\leq\left(\int_{\mathbb{R}^N}h(x)|u_n|^{2^*}\varphi_{0,\varepsilon} \, dx\right)^{\frac{\alpha}{2^*}}\left(\int_{\mathbb{R}^N}h(x)|v_n|^{2^*}\varphi_{0,\varepsilon} \, dx\right)^{\frac{\beta}{2^*}}.
\end{split}
\end{equation}
Then, because of \eqref{con-comp} and \eqref{hypH}, we get
\begin{equation*}
\begin{split}
\lim\limits_{n\to+\infty}\int_{\mathbb{R}^N}h(x)|u_n|^{2^*}\varphi_{0,\varepsilon}\,dx=&\int_{\mathbb{R}^N}h(x)|\tilde{u}|^{2^*}\varphi_{0,\varepsilon}\,dx+\rho_0h(0)\\
\leq&\int_{|x|\leq\varepsilon}h(x)|\tilde{u}|^{2^*} \, dx
\end{split}
\end{equation*}
and
\begin{equation*}
\begin{split}
\lim\limits_{n\to+\infty}\int_{\mathbb{R}^N}h(x)|v_n|^{2^*}\varphi_{0,\varepsilon}\, dx&=\int_{\mathbb{R}^N}h(x)|\tilde{v}|^{2^*}\varphi_{0,\varepsilon}\, dx+\overline{\rho}_0h(0)\\
&\leq\int_{|x|\leq\varepsilon}h(x)|\tilde{v}|^{2^*}\, dx.
\end{split}
\end{equation*}
Therefore, we conclude
\begin{equation*}
\begin{split}
\lim\limits_{\varepsilon\to0}\limsup\limits_{n\to+\infty}\int_{\mathbb{R}^N}h(x)|u_n|^\alpha |v_n|^\beta\varphi_{0,\varepsilon}\, dx&\leq\lim\limits_{\varepsilon\to0}\left(\int_{|x|\leq\varepsilon}h(x)|\tilde{u}|^{2^*}\, dx\right)^{\frac{\alpha}{2^*}}\left(\int_{|x|\leq\varepsilon}h(x)|\tilde{v}|^{2^*}\, dx\right)^{\frac{\beta}{2^*}}\\
&=0.
\end{split}
\end{equation*}
Since $\lim\limits_{|x|\to+\infty}h(x)=0$, the proof of \eqref{radial:atinf} follows in a similar way.
\end{proof}

For the non-radial case, we can prove the PS condition under the assumption that $\nu$ is sufficiently small.
\begin{lemma}\label{lemcritic2}
Suppose $\alpha+\beta=2^*$ and \eqref{hypH}. Let $\{(u_n,v_n)\} \subset \mathbb{D}$ be  a PS sequence for $\mathcal{J}_\nu$ at level $c\in\mathbb{R}$ such that \eqref{hyplemmaPS2} holds. Then, there exists $\tilde{\nu}>0$ such that, for every $\nu\leq \tilde{\nu}$, we have $(u_n,v_n)\to(\tilde{u},\tilde{v}) \in \mathbb{D}$ up to subsequence.
\end{lemma}

\begin{proof}
Note that we only need to take into account concentration phenomena at points $x_j\neq0, \infty$, since concentration at these points can be excluded by similar arguments to those of Lemma~\ref{lemcritic}. Moreover, without loss of generality, we can assume that the index $j\in\mathfrak{J}\cap\mathfrak{K}$ since, otherwise, it can be proved as before that
\begin{equation*}
\lim\limits_{\varepsilon\to0}\limsup\limits_{n\to+\infty}\int_{\mathbb{R}^N}h(x)|u_n|^\alpha |v_n|^\beta\varphi_{j,\varepsilon} \, dx=0,
\end{equation*}
where $\varphi_{j,\varepsilon}(x)$ is a cut-off function centered at $x_j\in\mathbb{R}^N$ defined as in \eqref{cutoff}. Hence, no concentration can ocurr for $x_j\in\mathbb{R}^N$ with $j\in\mathfrak{J}$ and $j\notin\mathfrak{K}$ or $x_k\in\mathbb{R}^N$ with $k\notin\mathfrak{J}$ and $k\in\mathfrak{K}$.\newline
Then, assuming $j\in\mathfrak{J}\cap\mathfrak{K}$ and testing $\mathcal{J}_{\nu}'(u_n,v_n)$ with $(u_n\varphi_{j,\varepsilon},0)$ we get
\begin{equation}\label{limitu}
\begin{split}
0&=\lim\limits_{n\to+\infty}\left\langle \mathcal{J}_{\nu}'(u_n,v_n){\big|}(u_n\varphi_{j,\varepsilon},0)\right\rangle\\
&=\lim\limits_{n\to+\infty}\left(\int_{\mathbb{R}^N}|\nabla u_n|^2\varphi_{j,\varepsilon}dx+\int_{\mathbb{R}^N}u_n\nabla u_n\nabla\varphi_{j,\varepsilon}\,dx-\lambda_1\int_{\mathbb{R}^N}\frac{u_n^2}{|x|^2}\varphi_{j,\varepsilon}dx\right.\\
&\mkern+80mu-\left.\int_{\mathbb{R}^N}|u_n|^{2^*}\varphi_{j,\varepsilon}dx-\alpha\nu\int_{\mathbb{R}^N}h(x)|u_n|^\alpha |v_n|^\beta\varphi_{j,\varepsilon}\, dx\right),
\end{split}
\end{equation}
and testing $\mathcal{J}_{\nu}'(u_n,v_n)$ with $(0,v_n\varphi_{j,\varepsilon})$ we get
\begin{equation}\label{limitv}
\begin{split}
0&=\lim\limits_{n\to+\infty}\left\langle \mathcal{J}_{\nu}'(u_n,v_n)\big|(0,v_n\varphi_{j,\varepsilon})\right\rangle\\
&=\lim\limits_{n\to+\infty}\left(\int_{\mathbb{R}^N}|\nabla v_n|^2\varphi_{j,\varepsilon}dx+\int_{\mathbb{R}^N}v_n\nabla v_n\nabla\varphi_{j,\varepsilon}dx-\lambda_2\int_{\mathbb{R}^N}\frac{v_n^2}{|x|^2}\varphi_{j,\varepsilon}dx\right.\\
&\mkern+80mu-\left.\int_{\mathbb{R}^N}|v_n|^{2^*}\varphi_{j,\varepsilon}dx-\nu\beta\int_{\mathbb{R}^N}h(x)|u_n|^\alpha |v_n|^\beta\varphi_{j,\varepsilon}\, dx\right).
\end{split}
\end{equation}
Then, since $h\in L^{\infty}(\mathbb{R}^N)$, using \eqref{ineq:holder} we find, for some constant $\tilde{C}>0$,
\begin{equation}\label{conuv}
\lim\limits_{\varepsilon\to0}\limsup\limits_{n\to+\infty}\int_{\mathbb{R}^N}h(x)|u_n|^\alpha |v_n|^\beta\varphi_{j,\varepsilon}\, dx\leq \tilde{C}\rho_j^{\frac{\alpha}{2^*}}\overline{\rho}_j^{\frac{\beta}{2^*}}.
\end{equation}
Hence, letting $\varepsilon\to0$, from \eqref{limitu}, \eqref{limitv} and \eqref{conuv} we get
\begin{equation*}
\begin{split}
\mu_j-\rho_j-\nu\alpha\tilde{C}\rho_j^{\frac{\alpha}{2^*}}\overline{\rho}_j^{\frac{\beta}{2^*}}\leq0,\\
\overline{\mu}_j-\overline{\rho}_j-\nu\beta\tilde{C}\rho_j^{\frac{\alpha}{2^*}}\overline{\rho}_j^{\frac{\beta}{2^*}}\leq0.
\end{split}
\end{equation*}
Then, by \eqref{ineq:sobcon}, we find
\begin{equation*}
S\left(\rho_j^{\frac{\alpha}{2^*}}+\overline{\rho}_j^{\frac{\beta}{2^*}}\right)\leq\rho_j+\overline{\rho}_j+2^*\nu \tilde{C}\rho_j^{\frac{\alpha}{2^*}}\overline{\rho}_j^{\frac{\beta}{2^*}}.
\end{equation*}
Therefore,
\begin{equation*}
S\left(\rho_j+\overline{\rho}_j\right)^{\frac{2}{2^*}}\leq(\rho_j+\overline{\rho}_j)(1+2^*\nu \tilde{C}).
\end{equation*}
As a consequence, we have two options either $\rho_j+\overline{\rho}_j=0$ or $\displaystyle\rho_j+\overline{\rho}_j\ge\left(\frac{S}{1+2^*\nu \tilde{C}}\right)^{\frac{N}{2}}$. In case of having concentration, arguing as in Lemma~\ref{lemmaPS2}, we have
\begin{equation*}
\begin{split}
 c&\ge\left(\frac{1}{2}-\frac{1}{\alpha+\beta}\right)(\mu_j+\overline{\mu}_j) + \left(\frac{1}{\alpha+\beta}-\frac{1}{2^*}   \right)  (\rho_j+\overline{\rho}_j)\\
&\ge \mathcal{S}\left(\frac{1}{2}-\frac{1}{\alpha+\beta}\right) (\rho_j+\overline{\rho}_j)^{\frac{2}{2^*}} + \left(\frac{1}{\alpha+\beta}-\frac{1}{2^*} \right) (\rho_j+\overline{\rho}_j)   \\
& \ge\frac{1}{N}\left(\frac{\mathcal{S}}{1+2^*\nu \tilde{C}}\right)^{\frac{N}{2}} .
\end{split}
\end{equation*}
Thus, for $\nu>0$ small enough, we conclude
\begin{equation*}
c\ge\frac{1}{N}\left(\frac{\mathcal{S}}{1+2^*\nu \tilde{C}}\right)^{\frac{N}{2}}\ge\frac{1}{N}\left(\min\{\mathcal{S}(\lambda_1),\mathcal{S}(\lambda_2)\}\right)^{\frac{N}{2}},
\end{equation*}
and we reach a contradiction with the hypothesis on the energy level $c$.
\end{proof}

\

\section{Main Results}\label{section:main}
This section is devoted to prove the main theorems of the work concerning existence of bound and ground states to system \eqref{system:alphabeta}.
In the following, we will use the next hypotheses.
\begin{equation}\tag{C}\label{alternative}
 \mbox{ Either } 2< \alpha+\beta < 2^* \qquad  \mbox{ or } \qquad \alpha+\beta=2^*  \mbox{ and } h  \mbox{ is radial and satisfies \eqref{hypH} }
\end{equation}

\

Our first result dealing with ground states is for $\nu$ sufficiently large, for which system \eqref{system:alphabeta} admits a positive ground state solution.
\begin{theorem}\label{thm:nugrande}
Assume \eqref{alternative}. Then there exists $\overline{\nu}>0$,  then system \eqref{system:alphabeta} admits a positive ground state solution $(\tilde{u},\tilde{v}) \in \mathbb{D}$ for $\nu>\overline{\nu}$.
\end{theorem}

\begin{proof}
Set $(u,v)\in\mathbb{D}\setminus (0,0)$, then there exists a constant $t=t_{(u,v)}$ such that $(tu,tv) \in \mathcal{N}_\nu$. More precisely, $t_{(u,v)}$ can be defined as the unique solution to the algebraic equation \eqref{normH}.

Due to the fact $\alpha+\beta>2$, then $t=t_\nu\to 0$ as $\nu\to+\infty$. Since $(t_\nu u,t_\nu v) \in \mathcal{N}_\nu$ and $\alpha+\beta\leq 2^*$, by using \eqref{normH}, one has
$$
\lim_{\nu \to+\infty} t_\nu \nu = \dfrac{\|(u,v)\|_\mathbb{D}^2}{\int_{\mathbb{R}^N} h(x) |u|^{\alpha} |v|^{\beta}  \, dx},
$$
for some $(u,v)\in \mathbb{D}$. On the other hand,
$$
\mathcal{J}_\nu(t_\nu u ,t_\nu v)=\left(\frac{1}{2}-\frac{1}{\alpha+\beta} +o(1)  \right) t_\nu^2 \|(u,v)\|^2_\mathbb{D}.
$$
Therefore, we conclude that
\beq\label{minimumlevel}
\tilde{c}_\nu=\inf_{u,v \in \mathcal{N}_\nu} \mathcal{J}_\nu (u,v)< \min \{\mathcal{J}_\nu(z_\mu^{\lambda_1},0),\mathcal{J}_\nu(0,z_\mu^{\lambda_2}) \}=\frac{1}{N}\min \{ \mathcal{S}(\lambda_1),\mathcal{S}(\lambda_2)\}^{\frac{N}{2}},
\eeq
for some $\nu>\overline{\nu}$ where $\overline{\nu}$ is sufficiently large.

For the subcritical regime $\alpha+\beta < 2^*$, by Lemma~\ref{lemmaPS2} there exists $(\tilde{u},\tilde{v}) \in \mathbb{D}$ such that $\mathcal{J}_\nu(\tilde{u},\tilde{v})=\tilde{c}_\nu$.\newline
Moreover, observe that
$$
\mathcal{J}_\nu(|\tilde{u}|,|\tilde{v}|)= \mathcal{J}_\nu(\tilde{u},\tilde{v}),
$$
so we can assume that $\tilde{u}\ge 0$ and $\tilde{v}\ge 0$ in $\mathbb{R}^N$. Applying the classical regularity results, $\tilde{u}$ and $\tilde{v}$ are smooth in $\R^N\setminus\{0\}$. In particular, $\tilde{u}\not \equiv 0$ and $\tilde{v}\not \equiv 0$. Otherwise, if $\tilde{u}\equiv 0$, then $\tilde{v}\ge 0 $ and $\tilde{v}$ verifies \eqref{vtildeq}, which implies that $\tilde{v}=z_\mu^{\lambda_2}$, a contradiction with \eqref{minimumlevel}. Using a similar argument, if $\tilde{v}\equiv 0$, one obtains a contradiction with \eqref{minimumlevel}. Consequently, by the maximum principle in $\R^N\setminus\{0\}$, we derive the existence of a ground state $(\tilde{u},\tilde{v}) \in \mathcal{N}_\nu$ such that $\tilde{u}> 0$ and $\tilde{v}> 0$ in $\mathbb{R}^N\setminus\{0\}$.

\
Using Lemma~\ref{lemcritic} and arguing as above, the same conclusion holds for the critical regime $\alpha+\beta=2^*$, namely, we conclude the existence of a positive ground state $(\tilde{u},\tilde{v})$.
\end{proof}

\

As the expressions \eqref{Jzeta} and \eqref{Slambda} show, the order between the energy levels of the \textit{semi-trivial} solutions is determined by the relation of the parameters $\lambda_1$ and $\lambda_2$. Actually, if $\lambda_1\ge \lambda_2$, the minimum level between both solutions corresponds to the couple $(z_\mu^{\lambda_1},0)$ which is a saddle point under certain assumptions on $\beta$ and $\nu$. Alternatively, if $\lambda_2\ge \lambda_1$, the minimum energy level corresponds to the couple $(0,z_\mu^{\lambda_2})$ which may be a saddle point under certain hypotheses on $\alpha$ and $\nu$. Therefore, for both situations there exists a positive ground state for \eqref{system:alphabeta}.

\begin{theorem}\label{thm:lambdaground}
Assume $\alpha+\beta=2^*$ and $\nu$ small or \eqref{alternative}. If one of the following alternatives holds:
\begin{itemize}
\item[i)] $\lambda_1\ge \lambda_2$ and either $\beta=2$ and $\nu$ large enough or $\beta<2$,
\item[ii)] $\lambda_2\ge \lambda_1$ and either $\alpha=2$ and $\nu$ large enough or $\alpha<2$,
\end{itemize}
then system \eqref{system:alphabeta} admits a positive ground state $(\tilde{u},\tilde{v})\in\mathbb{D}$.

In particular, if $\max\{\alpha,\beta\}<2$ or $\max\{\alpha,\beta\}\le2$ with $\nu$ sufficiently large, then system \eqref{system:alphabeta} admits a positive ground state $(\tilde{u},\tilde{v})\in\mathbb{D}$.
\end{theorem}

\begin{remark}
In the previous result, notice that if $\alpha+\beta=2^*$ and $\nu$ is small, the case $\beta=2$ and $\lambda_1\ge \lambda_2$  or $\alpha=2$ and $\lambda_2\ge \lambda_1$  do not hold by the previous assumptions since they require $\nu$ large.
\end{remark}

\begin{proof}
Let us start by proving the thesis assuming alternative i). Under one of the hypotheses, Proposition~\ref{thmsemitrivialalphabeta}, states that $(z_\mu^{\lambda_1},0)$ is a saddle point of $\mathcal{J}_\nu$ on $\mathcal{N}_\nu$. Moreover, due to  $\lambda_1\ge \lambda_2$
$$
\tilde{c}_\nu<\mathcal{J}_\nu(z_\mu^{\lambda_1},0)=\frac{1}{N}\mathcal{S}^{\frac{N}{2}}(\lambda_1)=\frac{1}{N}\min \{\mathcal{S}(\lambda_1),\mathcal{S}(\lambda_2)\}^{\frac{N}{2}},
$$
 where $\tilde{c}_\nu $ is defined by \eqref{ctilde}. Therefore, for subcritical case $\alpha+\beta<2^*$, by Lemma~\ref{lemmaPS2} there exists $(\tilde{u},\tilde{v})\in\mathcal{N}_\nu$  such that $\tilde{c}_\nu=\mathcal{J}_\nu(\tilde{u},\tilde{v})$. Arguing by contradiction as in Theorem~\ref{thm:nugrande}, one gets that $\tilde{u},\tilde{v} \ge 0$ and $(\tilde{u},\tilde{v}) \neq  (0,0)$. Finally, the maximum principle in $\R^N\setminus\{0\}$ allows us to conclude that $(\tilde{u},\tilde{v})$ is a positive ground state of \eqref{system:alphabeta}.

\

For the critical case $\alpha+\beta=2^*$, we obtain the existence of a positive ground state $(\tilde{u},\tilde{v})$ of  \eqref{system:alphabeta}, by using Lemma~\ref{lemcritic} in case that $h$ is radial and Lemma~\ref{lemcritic2} for $\nu$ small.

Repeating an analogous argument, we can deduce the same conclusion for alternative ii). If we suppose iii), it is immediate that
$$
\tilde{c}_\nu<\frac{1}{N}\min \{\mathcal{S}(\lambda_1),\mathcal{S}(\lambda_2)\}^{\frac{N}{2}},
$$
independently of the order between $\lambda_1$ and $\lambda_2$. Therefore, reasoning as before we infer the existence of a positive ground state.
\end{proof}

\
In case that $\lambda_2>\lambda_1$, the minimum energy provided by the {\it semi-trivial} solutions corresponds to the energy of $(0,z_\mu^{\lambda_2}) $. Moreover if either $\alpha>2$ or $\alpha=2$ with $\nu$ sufficiently small, the \textit{semi-trivial} solution is indeed a local minimum. In this context, the following result shows that for $\nu>0$ sufficiently small this couple realizes as a  ground state solution for \eqref{system:alphabeta}. Under analogous hypotheses, we can establish the same conclusion for $(z_\mu^{\lambda_1},0)$ in case that $\lambda_1\ge \lambda_2$.

The following result is an adaptation of \cite[Theorem~3.4]{AbFePe}. However, that result requires that $\alpha\ge 2$ and $\beta\ge 2$. In the subsequent, we prove that it is enough that one of the exponents will be bigger than $2$, namely $\max\{\alpha,\beta\}\ge 2$.

\begin{theorem}\label{thm:groundstatesalphabeta}
Assume $\alpha+\beta=2^*$ and $\nu$ small or \eqref{alternative}. Then the following holds:

\begin{itemize}
\item[i)] If $\alpha\ge 2$ and $\lambda_2>\lambda_1$, then there exists $\tilde{\nu}>0$ such that for any $0<\nu<\tilde{\nu}$  the couple $(0,z_\mu^{\lambda_2})$ is the ground state of \eqref{system:alphabeta}.

\item[ii)]  If $\beta\ge 2$ and $\lambda_1>\lambda_2$, then there exists $\tilde{\nu}>0$ such that for any $0<\nu<\tilde{\nu}$  the couple $(z_\mu^{\lambda_1},0)$ is the ground state of \eqref{system:alphabeta}.

\item[iii)] In particular, if $\alpha,\beta \ge 2$, then there exists $\tilde{\nu}>0$ such that for any $0<\nu<\tilde{\nu}$, the couple $(0, z_\mu^{\lambda_2})$ is a ground state of \eqref{system:alphabeta} if $\lambda_2\ge \lambda_1$ and $(z_\mu^{\lambda_1},0)$ is a ground state otherwise.
\end{itemize}
\end{theorem}

\begin{proof}
Let us show i). By Proposition~\ref{thmsemitrivialalphabeta}, $(0,z_\mu^{\lambda_2})$ is a local minimum for $\nu$ small enough. Next, assume, by contradiction, that there exists $\{\nu_n\} \searrow 0$ such that $\tilde{c}_{\nu_n} <  \mathcal{J}_{\nu_n} (0,z_\mu^{\lambda_2})$. Since $\lambda_2>\lambda_1$, then
\beq\label{groundstates1}
\tilde{c}_{\nu_n}< \frac{1}{N} \min \{ \mathcal{S}(\lambda_1), \mathcal{S}(\lambda_2) \}^{\frac{N}{2}}= \frac{1}{N} \mathcal{S}^{\frac{N}{2}}(\lambda_2),
\eeq
where $\tilde{c}_{\nu_n}$ is defined in \eqref{ctilde} with $\nu=\nu_n$. In case of $\alpha+\beta<2^*$, because of Lemma~\ref{lemmaPS2}, the PS condition holds at level $\tilde{c}_{\nu_n}$. If $\alpha+\beta=2^*$, we apply Lemma~\ref{lemcritic} for $h$ radial and Lemma~\ref{lemcritic2} for $\nu$ small, to obtain the same conclusion.

Therefore, there exists $(\tilde{u}_n,\tilde{v}_n) \in \mathbb{D}$ such that $\tilde{c}_{\nu_n}=\mathcal{J}_{\nu_n} (\tilde{u}_n,\tilde{v}_n)$. Since $\mathcal{J}_{\nu_n} (\tilde{u}_n,\tilde{v}_n)=\mathcal{J}_{\nu_n} (|\tilde{u}_n|,|\tilde{v}_n|)$, we can suppose that $\tilde{u}_n \ge 0$ and $\tilde{v}_n \ge 0$.

Moreover, as we proved in previous results, is not difficult to show that $\tilde{u}_n \not \equiv 0$ and $\tilde{v}_n \not \equiv 0$ in $\mathbb{R}^N$. Otherwise, it would contradict \eqref{groundstates1}. Using the maximum principle in $\mathbb{R}^N\setminus \{0\}$, one can conclude that actually $\tilde{u}_n > 0$ and $\tilde{v}_n > 0$ in $\mathbb{R}^N\setminus \{0\}$.\newline
Next, define the following integral quantities
\begin{equation*}
\sigma_{1,n}=\int_{\mathbb{R}^N} \tilde{u}_n^{2^*} \, dx \qquad \mbox{ and } \qquad \sigma_{2,n}=\int_{\mathbb{R}^N} \tilde{v}_n^{2^*} \, dx .
\end{equation*}
Note that, by \eqref{Nnueq2}, we have
\beq\label{groundstates3}
\tilde{c}_{\nu_n} = \mathcal{J}_{\nu_n} (\tilde{u}_n,\tilde{v}_n) = \frac{1}{N} \left(  \sigma_{1,n} + \sigma_{2,n}\right)  + \nu_n \left( \frac{\alpha+\beta-2}{2}  \right)\int_{\mathbb{R}^N} h(x)  \,  \tilde{u}_n^{\alpha} \,  \tilde{v}_n^\beta \, dx  .
\eeq
Combining \eqref{groundstates1} and \eqref{groundstates3}, we deduce that
\beq\label{groundstates4}
\sigma_{1,n}+\sigma_{2,n}<\mathcal{S}^{\frac{N}{2}}(\lambda_2).
\eeq
Now we take advantage from the fact that the pair $(\tilde{u}_n,\tilde{v}_n)$ is a solution to \eqref{system:alphabeta}. Working with the first equation and exploiting the definition of $\mathcal{S}(\lambda)$ given in \eqref{Slambda}, we obtain
\beq\label{groundstates45}
\mathcal{S}(\lambda_1) (\sigma_{1,n})^{\frac{N-2}{N}} \leq \sigma_{1,n} +  \nu_n \alpha  \int_{\mathbb{R}^N} h(x)  \,  \tilde{u}_n^{\alpha} \,  \tilde{v}_n^{\beta} \, dx  .
\eeq
Subsequently, applying H\"older's inequality and \eqref{groundstates4}, we get
$$
\int_{\mathbb{R}^N} h(x)  \,  \tilde{u}_n^{\alpha} \,  \tilde{v}_n^{\beta}  \, dx    \leq \|h\|_{L^{\infty}}  \left( \int_{\mathbb{R}^N} \tilde{u}_n^{2^*}  \, dx   \right)^{\frac{\alpha}{2^*}}\left( \int_{\mathbb{R}^N} \tilde{v}_n^{2^*}  \, dx    \right)^{\frac{\beta}{2^*}} \leq \|h\|_{L^{\infty}}  (\mathcal{S}(\lambda_2))^{\beta \frac{N-2}{4}} (\sigma_{1,n})^{\frac{\alpha}{2}\frac{N-2}{N}}.
$$
and, hence, from \eqref{groundstates45}, it follows that
\begin{equation*}
\mathcal{S}(\lambda_1) (\sigma_{1,n})^{\frac{N-2}{N}} < \sigma_{1,n} + \nu_n\alpha   \|h\|_{L^{\infty}}  (\mathcal{S}(\lambda_2))^{\beta \frac{N-2}{4}} (\sigma_{1,n})^{\frac{\alpha}{2}\frac{N-2}{N}}.
\end{equation*}
Since $\lambda_2> \lambda_1$, then there exists $\varepsilon>0$ such that
\beq\label{groundstates6}
(1-\varepsilon) \mathcal{S}^{\frac{N}{2}}(\lambda_1) \ge \mathcal{S}^{\frac{N}{2}}(\lambda_2).
\eeq
Next, because of Lemma~\ref{algelemma} with $\sigma=\sigma_{1,n}$, we infer the existence of $\tilde{\nu}=\tilde{\nu}(\varepsilon)>0$ such that
\beq\label{groundstates67}
\sigma_{1,n}> (1-\varepsilon) \mathcal{S}^{\frac{N}{2}}(\lambda_1) \qquad \mbox{ for any } 0<\nu_n<\tilde{\nu}.
\eeq
Because of \eqref{groundstates6}, one deduces that $\sigma_{1,n}>\mathcal{S}^{\frac{N}{2}}(\lambda_2)$, which contradicts \eqref{groundstates4}. Thus, we have proved that for $\nu$ sufficiently small
\beq\label{groundstates7}
\tilde{c}_\nu =   \frac{1}{N} \mathcal{S}^{\frac{N}{2}}(\lambda_2).
\eeq
Let $(\tilde{u},\tilde{v})$ be a minimizer of $\mathcal{J}_\nu$. Repeating the above argument, we can ensure that either $\tilde{u}\equiv 0$ or $\tilde{v}\equiv 0$. Obviously, if $v\equiv 0$, condition \eqref{groundstates7} would be violated. So $u\equiv 0$ and $\tilde{v}$ satisfies the equation
$$
-\Delta \tilde{v} - \lambda_2 \frac{\tilde{v}}{|x|^2}=|\tilde{v}|^{2^*-2}\tilde{v} \qquad \mbox{ in } \mathbb{R}^N.
$$
Let us show that the $\tilde{v}$ does not change sign and in fact $\tilde{v}= \pm z_{\mu}^{\lambda_2}$. Assume by contradiction that $\tilde{v}$ changes sign, then $\tilde{v}^{\pm} \not \equiv 0$ in $\mathbb{R}^N$. Since $(0,\tilde{v}) \in \mathcal{N}_\nu$, hence $(0,\tilde{v}^\pm) \in \mathcal{N}_\nu$. Moreover, due to \eqref{groundstates3}
$$
\tilde{c}_{\nu}= \mathcal{J}_\nu (0,\tilde{v}) = \frac{1}{N} \int_{\mathbb{R}^N} |\tilde{v}|^{2^*}  \, dx  = \frac{1}{N} \left( \int_{\mathbb{R}^N} (\tilde{v}^+)^{2^*}  \, dx  + \int_{\mathbb{R}^N}  |\tilde{v}^-|^{2^*}  \, dx  \right) > \mathcal{J}_\nu (0,\tilde{v}^+) \ge  \tilde{c}_{\nu},
$$
which is a contradiction. Hence, $(0,\pm z_{\mu}^{\lambda_2})$ is the minimizer of $\mathcal{J}_\nu$ in $\mathcal{N}_\nu$ if $\lambda_2>\lambda_1$. Even more, $(0,z_{\mu}^{\lambda_2})$ is a ground state to \eqref{system:alphabeta}.

\

In a similar way we can deduce ii). In case that $\alpha,\beta\ge 2$ we apply an analogous argument for the semi-trivial couple whose energy level is the minimum one.

\end{proof}

\

\begin{remark}
The interval for admissible parameters $\nu$ in Theorem~\ref{thm:groundstatesalphabeta} depends on $|\lambda_2-\lambda_1|$. For instance, if $\lambda_2>\lambda_1$, the range of $\varepsilon$ which verifies inequality \eqref{groundstates6} increases as $\lambda_2-\lambda_1$ increases. This implies that \eqref{groundstates67} is satisfied for a bigger range of parameters $\nu$ as a consequence of Lemma~\ref{algelemma}.
\end{remark}

\

In the following result we find bound states by a min-max procedure. More precisely, we shall show that the energy functional $\mathcal{J}^+_\nu$, introduced in \eqref{funct:SKdVp}, admits the \textit{Mountain--Pass--geometry} for parameters $\lambda_1,\lambda_2$ which verifies a separability condition, that allows us to separate the \textit{semi-trivial} energy levels in a suitable way.

\

\

\begin{theorem}\label{MPgeom}
Assume \eqref{alternative}. If
\begin{itemize}
\item[i)] Either
\beq\label{lamdasalphabeta}
\mbox{ $\alpha \ge 2 \,  $,  \quad  $ \lambda_2> \lambda_1$ \quad  and } \qquad  \frac{\Lambda_N-\lambda_2}{\Lambda_N-\lambda_1}>2^{-\frac{2}{N-1}},
\eeq
\item[ii)] or
\begin{equation*}
\mbox{  $\beta \ge 2 $,  \quad  $\lambda_1> \lambda_2$ \qquad and }\qquad  \frac{\Lambda_N-\lambda_1}{\Lambda_N-\lambda_2}>2^{-\frac{2}{N-1}},
\end{equation*}
\end{itemize}
then there exists $\tilde{\nu}>0$ such that for $0<\nu\le \tilde{\nu}$, the problem \eqref{system:alphabeta} admits a bound state of Mountain--Pass--type critical point.
\end{theorem}

\begin{proof}
We will prove the result only for assumption i), since the proof under assumption ii) follows analogously. The proof is divided into two steps. First, we prove that the energy functional  $\mathcal{J}_\nu^+\Big|_{\mathcal{N}^+_\nu} $ admits a Mountain--pass--geometry. Secondly, we show that the PS condition is guaranteed for the Mountain--Pass level. As a consequence, we can deduce the existence of $(\tilde{u},\tilde{v})\in\mathbb{D}$ which is a critical point of $\mathcal{J}_\nu^+$ and, therefore, a bound state of \eqref{system:alphabeta}.

\textbf{Step 1:}\hfill\newline
Let us define the set of paths that connects $(z_{\mu}^{\lambda_1},0)$ to $(0,z_{\mu}^{\lambda_2})$ continuously, namely
$$
\Psi_\nu = \left\{ \psi(t)=(\psi_1(t),\psi_2(t))\in C^0([0,1],\mathcal{N}^+_\nu): \, \psi(0)=(z_1^{\lambda_1},0) \mbox{ and } \, \psi(1)=(0,z_1^{\lambda_2})\right\},
$$
and the Mountain--Pass level
\begin{equation*}
c_{MP} = \inf_{\psi\in\Psi_\nu} \max_{t\in [0,1]} \mathcal{J}^+_{\nu} (\psi(t)).
\end{equation*}
Assumption \eqref{lamdasalphabeta} implies that
$$
\frac{2}{N} \mathcal{S}^{\frac N2}(\lambda_2) > \frac{1}{N} \mathcal{S}^{\frac N2}(\lambda_1).
$$
Moreover, by continuity and monotonicity of $\mathcal{S}(\lambda)$, we can take $\varepsilon>0$ small enough such that
\beq\label{MPgeomp0}
\frac{2}{N}(1-\varepsilon)\left( \frac{\mathcal{S}(\lambda_1)+\mathcal{S}(\lambda_2)}{2}\right)^{\frac{N}{2}}> \frac{2}{N} \mathcal{S}^{\frac{N}{2}}(\lambda_2)>\frac{1+\varepsilon}{N} \mathcal{S}^{\frac{N}{2}}(\lambda_1).
\eeq

Next, we claim that there exists $\tilde{\nu}=\tilde{\nu}(\varepsilon)>0$ such that, for every $0<\nu<\tilde{\nu}$,
\beq\label{claim1}
\max_{t\in[0,1]} \mathcal{J}_\nu^+ (\psi(t)) \ge \frac{2}{N}(1-\varepsilon)\left( \frac{\mathcal{S}(\lambda_1)+\mathcal{S}(\lambda_2)}{2}\right)^{\frac{N}{2}} \qquad \mbox{ with } \psi \in \Psi_\nu.
\eeq

Take $\psi=(\psi_1,\psi_2) \in \Psi_\nu$, then by the identity \eqref{Nnueqp}, we obtain that
\begin{equation}\label{MPgeomp1}
\begin{split}
 \int_{\mathbb{R}^N} &\left( |\nabla \psi_1(t)|^2 + |\nabla \psi_2(t)|^2  \right) \, dx -\lambda_1 \int_{\mathbb{R}^N} \dfrac{\psi_1^2(t)}{|x|^2}  \, dx   - \lambda_2 \int_{\mathbb{R}^N} \dfrac{\psi_2^2(t)}{|x|^2} \, dx  \vspace{0.3cm}\\
&= \int_{\mathbb{R}^N} \left( (\psi_1^+(t))^{2^*} + (\psi_2^+(t))^{2^*}  \right) \, dx  + \nu(\alpha+\beta) \int_{\mathbb{R}^N} h(x) (\psi_1^+(t))^\alpha (\psi_2^+(t))^\beta \, dx  ,
\end{split}
\end{equation}
and using \eqref{Nnueqp1}
\beq\label{MPgeomp2}
\begin{split}
\mathcal{J}^+_{\nu} (\psi(t)) =& \frac{1}{N} \left( \int_{\mathbb{R}^N} (\psi_1^+(t))^{2^*} + (\psi_2^+(t))^{2^*} \, dx \right)\\
  &+ \nu\left(\frac{\alpha+\beta-2}{2}\right) \int_{\mathbb{R}^N} h(x)  \,  (\psi_1^+(t))^\alpha (\psi_2^+(t))^\beta \, dx.
\end{split}
\eeq

Let us define $\sigma(t)=\left(\sigma_1(t),\sigma_2(t)\right)$ where $\sigma_j(t)=\int_{\mathbb{R}^N} (\psi_j^+(t))^{2^*} \, dx$ for $j=1,2$. Without loss of generality, let us suppose that $\sigma_j(t)\leq 2 \mathcal{S}^{\frac{N}{2}}(\lambda_1)$ for $t\in[0,1]$ and $j=1,2$. Otherwise, \eqref{claim1} is done.

On the other hand, by the definition of $\mathcal{S}(\lambda)$, from \eqref{MPgeomp1} we get
\begin{equation}\label{MPgeomp3}
\begin{split}
\mathcal{S}(\lambda_1)(\sigma_1(t))^{\frac{N-2}{N}}+\mathcal{S}(\lambda_2)(\sigma_2(t))^{\frac{N-2}{N}}
\leq& \int_{\mathbb{R}^N} \left( |\nabla \psi_1(t)|^2 + |\nabla \psi_2(t)|^2  \right) dx\\
&-\lambda_1 \int_{\mathbb{R}^N} \dfrac{\psi_1^2(t)}{|x|^2} \, dx  - \lambda_2 \int_{\mathbb{R}^N} \dfrac{\psi_2^2(t)}{|x|^2} \, dx \\
=&\sigma_1(t)+\sigma_2(t)\\
&+\nu (\alpha+\beta) \int_{\mathbb{R}^N} h(x) (\psi_1^+(t))^\alpha (\psi_2^+(t))^\beta dx.
\end{split}
\end{equation}

By H\"older's inequality, one can bound the last integral as
\begin{equation}\label{MPgeomp4}
\int_{\mathbb{R}^N} h(x) (\psi_1^+(t))^\alpha (\psi_2^+(t))^\beta \, dx \leq \nu \|h\|_{L^{\infty}}  (\sigma_1(t))^{\frac{\alpha}{2}\frac{N-2}{N}} (\sigma_2(t))^{\frac{\beta}{2}\frac{N-2}{N}}. \end{equation}

From the definition of $\psi$, we know that
$$
\sigma(0)=\left(\int_{\mathbb{R}^N} (z_1^{\lambda_1})^{2^*} \, dx,0\right) \quad \mbox{ and } \quad \sigma(1)=\left(0,\int_{\mathbb{R}^N} (z_1^{\lambda_2})^{2^*} \, dx \right).$$
Then by continuity of $\sigma$, there exists $\tilde{t}\in(0,1)$ such that $\sigma_1(\tilde{t})=\tilde{\sigma}=\sigma_2(\tilde{t})$. Taking $t=\tilde{t}$ in inequality \eqref{MPgeomp3} and applying \eqref{MPgeomp4}, we have that
$$
(\mathcal{S}(\lambda_1)+\mathcal{S}(\lambda_2)) \tilde{\sigma}^{\frac{N-2}{N}} \leq 2 \tilde{\sigma} + \nu (\alpha+\beta) \tilde{\sigma}^{\frac{\alpha+\beta}{2}\frac{N-2}{N}}.
$$

Since $\alpha+\beta>2$, Lemma~\ref{algelemma} ensures the existence of $\tilde{\nu}$ which depends on $\varepsilon$ such that
\begin{equation}\label{MPgeomp5}
\tilde{\sigma}\ge (1-\varepsilon)  \left(\frac{\mathcal{S}(\lambda_1)+\mathcal{S}(\lambda_2)}{2}\right)^{\frac{N}{2}} \qquad \mbox{ for every } 0<\nu\le\tilde{\nu}.
\end{equation}

As a result, from \eqref{MPgeomp2} and \eqref{MPgeomp5} we deduce that
$$
\max_{t\in[0,1]} \mathcal{J}^+_\nu(\psi(t)) \ge \frac{ \sigma_1(t)+ \sigma_2(t)}{N}  \ge\frac{2(1-\varepsilon) }{N} \left(\frac{\mathcal{S}(\lambda_1)+\mathcal{S}(\lambda_2)}{2}\right)^{\frac{N}{2}},
$$
which proves claim \eqref{claim1}. Moreover, by \eqref{MPgeomp0} and \eqref{claim1}, one can state that
\begin{equation}\label{MPgeomp6}
c_{MP}>\frac{(1+\varepsilon)}{N} \mathcal{S}^{\frac{N}{2}}(\lambda_1)=(1+\varepsilon)\mathcal{J}^+_\nu(z_1^{\lambda_1},0),
\eeq
which implies that the energy functional $\mathcal{J}^+_\nu$ exhibits a Mountain--Pass--geometry on $\mathcal{N}_\nu$.

\

\textbf{Step 2:}\hfill\newline Let us consider $\psi(t) =(\psi_1(t),\psi_2(t))=\left((1-t)^{1/2} z_1^{\lambda_1},t^{1/2}z_1^{\lambda_2} \right)$ for $t\in[0,1]$. By the definition of the Nehari manifold, there exists a positive function $\gamma:[0,1]\to(0,+\infty)$ such that the $\gamma \psi \in \mathcal{N}_\nu^+\cap \mathcal{N}_\nu$ for  $t\in[0,1]$. Observe that $\gamma(0)=\gamma(1)=1$.\newline
As above, let us define
$$
\sigma(t)=(\sigma_1(t),\sigma_2(t))=\left(\int_{\mathbb{R}^N}\left( \gamma \psi_1(t)\right)^{2^*} \, dx , \int_{\mathbb{R}^N} \left(\gamma \psi_2(t)\right)^{2^*} \, dx \right).
$$
By \eqref{normcrit}, we have that
\beq\label{Mpgeom7}
\sigma_1(0)=\int_{\mathbb{R}^N} (z_1^{\lambda_1})^{2^*} = \mathcal{S}^{\frac{N}{2}}(\lambda_1) \quad \mbox{ and }\quad \sigma_2(1)=\int_{\mathbb{R}^N} (z_1^{\lambda_2})^{2^*} = \mathcal{S}^{\frac{N}{2}}(\lambda_2).
\eeq
Due to $\gamma\psi(t)\in\mathcal{N}^+_\nu\cap \mathcal{N}_\nu$, using \eqref{normH}, we get
\begin{equation*}
\begin{split}
\left\|\left((1-t)^{1/2} z_1^{\lambda_1},t^{1/2}z_1^{\lambda_2} \right)\right\|^2_\mathbb{D} &= \gamma^{2^*-2}(t) \left((1-t)^{2^*/2} \sigma_1(0) + t^{2^*/2} \sigma_2(1) \right) \\
&\mkern+20mu + \nu (\alpha+\beta)\gamma (t)(1-t)t^{1/2} \int_{\mathbb{R}^N} h(x) (z_1^{\lambda_1})^\alpha (z_1^{\lambda_2})^{\beta}dx,
\end{split}
\end{equation*}
and, therefore,
\beq\label{gammabound}
\gamma^{2^*-2}(t)< \dfrac{\left|\left|\left( \psi_1(t), \psi_2(t) \right)\right|\right|^2_\mathbb{D}}{\int_{\mathbb{R}^N} (\psi_1(t))^{2^*}+(\psi_2(t))^{2^*} \, dx}= \dfrac{(1-t) \sigma_1(0) + t \sigma_2(1)}{(1-t)^{2^*/2} \sigma_1(0) + t^{2^*/2} \sigma_2(1)},
\eeq
for every $t\in(0,1)$. On the other hand, by definition of $\gamma$, \eqref{gammabound} and \eqref{Nnueq}, it follows that
\begin{equation}\label{Jgammapsibound}
\begin{split}
\mathcal{J}_\nu^+(\gamma \psi(t)) =& \left( \frac{1}{2}-\frac{1}{\alpha+\beta} \right)\|\gamma\psi(t) \|^2_\mathbb{D}\\
&+ \left(\frac{1}{\alpha+\beta}- \frac{1}{2^*} \right) \gamma^{2^*}(t) \left(\int_{\mathbb{R}^N} (\psi_1(t))^{2^*}+(\psi_2(t))^{2^*}  \, dx \right) \vspace{0.3cm} \\
 =&\ \gamma^2(t)  \left( \frac{1}{2}-\frac{1}{\alpha+\beta} \right)  \left[(1-t) \sigma_1(0) + t \sigma_2(1) \right] \\
 &+ \left(\frac{1}{\alpha+\beta}- \frac{1}{2^*} \right) \gamma^{2^*}(t) \left[ (1-t)^{2^*/2} \sigma_1(0) + t^{2^*/2} \sigma_2(1)\right]\\
<& \frac{\gamma^2(t)}{N} \left[(1-t) \sigma_1(0) + t \sigma_2(1) \right] .
\end{split}
\end{equation}
Then, by \eqref{gammabound} and \eqref{Jgammapsibound}, we deduce that, for $0<t<1$,
$$
\mathcal{J}_\nu^+(\gamma \psi(t)) < g(t)\vcentcolon= \dfrac{(1-t) \sigma_1(0) + t \sigma_2(1)}{N}  \left[ \dfrac{(1-t) \sigma_1(0) + t \sigma_2(1)}{(1-t)^{2^*/2} \sigma_1(0) + t^{2^*/2} \sigma_2(1)} \right]^{\frac{N-2}{2}}.
$$
Since function $g(t)$ attains its maximum at $t=\frac{1}{2}$ and due to \eqref{Mpgeom7}
$$
g\left(\frac{1}{2}\right)= \dfrac{ \sigma_1(0) + \sigma_2(1)}{N}=\dfrac{\mathcal{S}^{\frac{N}{2}}(\lambda_1)+\mathcal{S}^{\frac{N}{2}}(\lambda_2)}{N},
$$
we conclude
$$
c_{MP} \leq \max_{t\in[0,1]} \mathcal{J}_\nu^+(\gamma \psi(t))< \dfrac{\mathcal{S}^{\frac{N}{2}}(\lambda_1)+\mathcal{S}^{\frac{N}{2}}(\lambda_2)}{N}.
$$
Finally, because of the separability condition \eqref{lamdasalphabeta} and \eqref{MPgeomp6}, if $\lambda_2>\lambda_1$ we have
$$
\frac{\mathcal{S}^{\frac{N}{2}}(\lambda_2)}{N}<\frac{\mathcal{S}^{\frac{N}{2}}(\lambda_1)}{N}<c_{MP}<\frac{1}{N}\left( \mathcal{S}^{\frac{N}{2}}(\lambda_1)+\mathcal{S}^{\frac{N}{2}}(\lambda_2) \right)<3\frac{\mathcal{S}^{\frac{N}{2}}(\lambda_2)}{N}.
$$
Thus, the energy level $c_{MP}$ satisfies the hypotheses of Lemma~\ref{lemmaPS1} and Lemma~\ref{lemcritic}. Then, by the Mountain--Pass Theorem, there exists a sequence $\left\{ (u_n,v_n) \right\} \subset \mathcal{N}^+_\nu$ such that
$$
\mathcal{J}^+(u_n,v_n ) \to c_\nu \qquad \mathcal{J}^+|_{\mathcal{N}
^+_\nu}(u_n,v_n ) \to 0.
$$
If $\alpha+\beta<2^*$, because of analogous versions of Lemmas~\ref{lemma:PSNehari} and \ref{lemmaPS1} for $\mathcal{J}^+_\nu$, we have $\left\{ u_n,v_n \right\} \to ( \tilde{u},\tilde{v} )$. Actually, $(\tilde{u},\tilde{v} )$ is a critical point of $\mathcal{J}_\nu$ on $\mathcal{N}_\nu$. It follows that is a critical point of $\mathcal{J}_\nu$ defined in $\mathbb{D}$. Moreover, $\tilde{u},\tilde{v} \ge 0$ in $\mathbb{R}^N$ and by maximum principle in $\mathbb{R}^N\setminus \{0\}$ we can conclude that they are strictly positive.
\
For assumptions ii), we would apply Lemma~\ref{lemmaPS1a} to guarantee the PS condition.

In case that $\alpha+\beta=2^*$, we follow the same approach for the compactness of the PS sequence, using Lemma~\ref{lemcritic}.

\end{proof}

\begin{remark}
Comparing to Theorem~3.8 and Theorem~4.2 in \cite{AbFePe}, which works in dimension $N=3$ and $N=3,4$ respectively, Theorem~\ref{MPgeom} works until dimension $5$. In some sense, the wider range of the exponents allows the dimension to be bigger.
\end{remark}

\

\begin{remark}Let us stress that:
\begin{itemize}
\item[i)] If $\lambda_1=\lambda_2$, then the \textit{non-semi-trivial solution} provided by Theorem~\ref{thm:lambdaground} is a ground state. As it is observed in \cite{AbFePe}, $\mathcal{J}_\nu$ admits a Mountain--Pass geometry. However, it is not clear if the PS condition is satisfied for the Mountain--Pass level given in Theorem~\ref{MPgeom}. See Remark~3.10 in \cite{AbFePe} for further details.

\item[ii)] If $\lambda_1\leq\lambda_2$ and

\begin{equation*}
2^\frac{-2}{N-1}>\frac{\Lambda_N-\lambda_2}{\Lambda_N-\lambda_1}.
\end{equation*}
It remains an open problem the geometry of the energy functional and if there exists other solutions apart from the semi-trivial solutions and the ground states for $\nu$ small enough.
\end{itemize}
\end{remark}

\vspace{1cm}

\begin{center}{\bf Acknowledgements}\end{center} This work has been supported by the Madrid Government (Comunidad de Madrid-Spain) under the Multiannual Agreement with UC3M in the line of Excellence of University Professors (EPUC3M23), and in the context of the V PRICIT (Regional Programme of Research and Technological Innovation).\\ The authors are partially supported by the Ministry of Economy and Competitiveness of Spain, under research project PID2019-106122GB-I00.


\begin{thebibliography}{99}

\bibitem{AA} J. Albert, J. Angulo Pava, {\it Existence and stability of ground-state solutions of a Schr\"odinger-KdV system}. Proc. Roy. Soc. Edinburgh Sect. A 133 (2003) 987--1029.

\bibitem{AbFePe} B. Abdellaoui, V. Felli, I. Peral, {\it Some remarks on systems of elliptic equations doubly critical in the whole $\mathbb{R}^N$}.
 Calc. Var. Partial Differential Equations, 34 (2009), no. 1, 97--137.


\bibitem{AC2}  A. Ambrosetti, E. Colorado, {\it Standing waves of some coupled nonlinear Schr\"odin\-ger equations}. J. Lond. Math.
Soc. 2 (75) (2007), no. 1, 67--82.

\bibitem{Akhmediev} N. Akhmediev, A. Ankiewicz, {\it  Partially coherent solitons on a finite background}. Phys. Rev.
Lett. 82, 2661--2664 (1999).


\bibitem{ACR} A. Ambrosetti, E. Colorado, D. Ruiz, {\it Multi-bump solitons to linearly coupled systems of nonlinear Schr\"odinger
equations}. Calc. Var. Partial Differential Equations  30 (2007), no. 1, 85--112.

\bibitem{BW} T. Bartsch, Z.-Q. Wang, {\it Note on ground states of nonlinear Schr\"odinger systems}. J. Partial Differential
Equations 19 (2006), no. 3, 200--207.

\bibitem{BVG} M.F. Bidaut-Veron, P.Grillot, {\it Singularities in elliptic systemswith absorption terms}. Ann. Sc. Norm.
Sup. Pisa Cl. Sci. (4) 28(4), 229–271 (1999).

\bibitem{ChenZou} Z. Chen, W. Zou, {\it Existence and symmetry of positive ground states for a doubly critical Schrödinger system}.
 Trans. Amer. Math. Soc. 367 (2015), no. 5, 3599--3646.

\bibitem{Col} E. Colorado, {\it On the existence of bound and ground states for a system of coupled nonlinear Schr\"odinger--Korteweg-de Vries Equations}, Adv. Nonlinear Anal. 6 (2017), no. 4, 407--426.

\bibitem{CoLSOr} E. Colorado, R. López-Soriano, A. Ortega, {\it Bound and ground states of coupled ``NLS-KdV" equations with Hardy potential and critical power}, in progress.

\bibitem{FigPerRos} D.G. De Figueiredo, I. Peral, J.D. Rossi, J, {\it The critical hyperbola for a Hamiltonian elliptic system with weights}. Ann. Mat. Pura Appl. (4) 187 (2008), no. 3, 531--545.

\bibitem{DFO} J.-P. Dias, M. Figueira, F. Oliveira, {\it Well-posedness and existence of bound states for a coupled Schr\"odinger-gKdV system.} Nonlinear Anal. 73 (2010), no. 8, 2686--2698.

\bibitem{Esry} B.D. Esry, C.H. Greene, J.P. Burke, Jr., J.L. Bohn, {\it Hartree-Fock Theory for Double Condensates}. Phys. Rev. Lett., 78, (1997), 3594--3597.

\bibitem{Frantz} D.J Frantzeskakis, {\it Dark solitons in atomic Bose–Einstein condensates: from theory to experiments}. J. Phys. A Math. Theory 43, 213001 (2010).

\bibitem{FO} M. Funakoshi, M. Oikawa, {\it The resonant Interaction between a Long Internal Gravity Wave and a Surface Gravity Wave Packet}. J. Phys. Soc. Japan. 52 (1983), no.1, 1982--1995.

\bibitem{Kivshar} Y.S. Kivshar, B. Luther-Davies, {\it Dark optical solitons: physics and applications}. Phys. Rep. 298, 81–197 (1998).

\bibitem{LinWei} T.-C. Lin, J. Wei, {\it Ground state of $N$ coupled nonlinear Schrödinger equations in $\mathbb{R}^n$ with $n\le 3$}.  Commun. Math. Phys. 255(3), 629–653 (2005).

\bibitem{Lions1} P.L. Lions, {\it The concentration-compactness principle in the calculus of variations. The limit case part 1}.
Rev. Matemática Iberoamericana, 1(1), 145--201 (1985).

\bibitem{Lions2} P.L. Lions, {\it The concentration-compactness principle in the calculus of variations. The limit case part 2}.
Rev. Matemática Iberoamericana, 1(2), 45--121 (1985).

\bibitem{LW} Z. Liu, Z.-Q. Wang, {\it Ground states and bound states of a nonlinear Schr\"odinger system}. Adv. Nonlinear Stud. 10 (2010), no. 1, 175--193.

\bibitem{MMP} L. Maia, E. Montefusco, B. Pellacci, {\it  Positive solutions for a weakly coupled nonlinear Schr\"odinger system}.
J. Differential Equations  229 (2006), no. 2, 743--767.

\bibitem{POMP} A. Pomponio, {\it Coupled nonlinear Schr\"odinger systems with potentials}. J. Differential Equations
 227 (2006), no. 1, 258--281.

\bibitem{SIR} B. Sirakov, {\it  Least energy solitary waves for a system of nonlinear Schr\"odinger equations in $\mathbb{R}^n$}. Comm. Math. Phys.  271 (2007), no. 1, 199--221.

 \bibitem{Smets} D. Smets, {\it  Nonlinear Schrödinger equations with Hardy potential and critical nonlinearities.} Trans. Amer. Math. Soc. 357 (2005), no. 7, 2909--2938.

 \bibitem{Stru} M. Struwe, {\it  A global compactness result for elliptic boundary value problems involving limiting nonlinearities}. Math. Z. 187 (1984), 511--517.


\bibitem{Terracini} S. Terracini, {\it On positive entire solutions to a class of equations with a singular coefficient and critical exponent}. Adv. Differential Equations, 1 (1996), no. 2, 241--264.

\bibitem{ZhongZou} X. Zhong, W. Zou, {\it Critical Schrödinger systems in $\mathbb{R}^N$ with indefinite weight and Hardy potential}.  Differential Integral Equations, 28 (2015), no. 1-2, 119--154.

\end{thebibliography}
\end{document}